\documentclass[11pt,a4paper]{scrartcl}

\usepackage{a4wide}
\usepackage[UKenglish]{babel}
\usepackage{latexsym}
\usepackage{amssymb}
\usepackage{mathrsfs}
\usepackage{graphicx}
\usepackage[latin1]{inputenc}
\usepackage{amsmath}
\usepackage{amsfonts}
\usepackage{amsthm}
\usepackage{url}
\usepackage{mathptmx}
\usepackage{helvet}
\usepackage{dsfont}
\usepackage{setspace}
\usepackage{enumerate}
\usepackage{hyperref}
\hypersetup{
  bookmarksopen=true,
  citebordercolor=0 1 0,
  linkbordercolor=1 0 0,
  menubordercolor=1 0 0,
  }

 \newcommand{\bfi}{\begin{fig}}
 \newcommand{\efi}{\end{fig}}

 \newcommand{\btab}{\begin{tab}}
 \newcommand{\etab}{\end{tab}}

 \newcommand{\barr}{\begin{array}}
 \newcommand{\earr}{\end{array}}

 \newcommand{\beqq}{\begin{equation}}
 \newcommand{\eeqq}{\end{equation}}

 \newcommand{\beao}{\begin{eqnarray*}}
 \newcommand{\eeao}{\end{eqnarray*}\noindent}

 \newcommand{\beam}{\begin{eqnarray}}
 \newcommand{\eeam}{\end{eqnarray}\noindent}

 \newcommand{\bdis}{\begin{displaymath}}
 \newcommand{\edis}{\end{displaymath}\noindent}

 \newcommand{\ben}{\begin{enumerate}} 
 \newcommand{\een}{\end{enumerate}} 

 \newcommand{\bali}{\begin{align}} 
 \newcommand{\eali}{\end{align}} 

 \newcommand{\bbn}{\mathbb{N}}
 \newcommand{\bbz}{\mathbb{Z}}
 \newcommand{\bbr}{\mathbb{R}}

 \newcommand{\scrb}{{\mathscr{B}}}

 \newcommand{\al}{{\alpha}}

 \newcommand{\bbE}{{\mathbb{E}}}
 
 \newcommand{\bbP}{{\mathbb{P}}}
 
\newtheorem{satz}{Theorem}[section]

\newtheorem{proposition}[satz]{Proposition}
\newtheorem{lemma}[satz]{Lemma}
\newtheorem{definition}[satz]{Definition}
\theoremstyle{definition}
\newtheorem{bemerkung}[satz]{Remark}
\newtheorem{beispiel}[satz]{Example}

\AtBeginDocument{%
}
\numberwithin{equation}{section}
\title{Stationarity and Geometric Ergodicity of BEKK Multivariate GARCH Models}
\author{Farid Boussama \thanks{D\'epartement Informatique et Math\'ematique, Universit\'e Montpellier 1, 39 Rue de l'Universit\'e, 34000 Montpellier, France. 
\emph{Email:} \ttfamily{boussama@univ-montp1.fr} } \and 
Florian Fuchs \thanks{TUM Institute for Advanced Study \& Zentrum Mathematik, Technische Universit\"at M\"unchen, Boltzmannstra\ss e 3, D-85748 Garching, Germany.
\emph{Email:} \ttfamily{ffuchs@ma.tum.de, www-m4.ma.tum.de}} \and 
Robert Stelzer \thanks{Institute of Mathematical Finance, Ulm University, Helmholtzstra\ss e 18, D-89081 Ulm, Germany. 
\emph{Email:} \ttfamily{robert.stelzer@uni-ulm.de, www.uni-ulm.de/mawi/finmath.html}}}

\date{}

\begin{document}
\maketitle
\begin{abstract}
Conditions for the existence of strictly stationary multivariate GARCH processes in the so-called BEKK parametrisation, which is the  most general form of multivariate
GARCH processes typically used in applications, and for their geometric ergodicity are obtained. The conditions are  that the driving noise is absolutely continuous with
respect to the Lebesgue measure and zero is in the interior of its support and that a certain matrix built from the GARCH coefficients has spectral radius smaller than one.

To establish the results semi-polynomial Markov chains are defined and analysed using algebraic geometry.
\end{abstract}
\vspace{0.5cm}
\noindent
\begin{tabbing}
\> Secondary: \, 60B99, \, 62M10, \,  91G70
\end{tabbing}
\vspace{0.5cm}
\section{Introduction}
Generalised autoregressive conditionally heteroskedastic (GARCH) processes (originally introduced by \cite{Bollerslev1986,Engle1982}) are heavily used in various areas of
applications for the modelling of heteroskedastic time series data. Very often one  has to model several interrelated time series with an appropriate multidimensional 
model. Since for multivariate GARCH processes the latent volatility process  needs to take values in the positive semi-definite matrices, as it has to correspond to a
covariance matrix at each point in time, the multivariate GARCH models are typically considerably more involved than the univariate one. For an overview over the various
multivariate GARCH models existing and their applications we refer to \cite{Bauwensetal2006,SilvennoinenTeraesvirta2009}. As always in time series modelling, existence
and uniqueness of stationary solutions, as well as convergence to the stationary solution is of high importance. While this is not so hard a question for some multivariate
GARCH specifications where -- as in the Constant-Conditional-Correlation (CCC) Model or its extensions -- the variances are modelled by univariate GARCH models, we
address this question here for the very general Baba, Engle, Kraft and Kroner (BEKK) GARCH model introduced in \cite{ENGMUL} where all (co)variances influence each
other in the time dynamics. The BEKK model is almost the most general multivariate GARCH model existing. Only the vec model also introduced in \cite{ENGMUL} is more
general, but all vec models not representable in the BEKK parametrisation are somewhat degenerate (see \cite{STEREL}). Moreover, the restrictions on the parameters
necessary to ensure a proper GARCH model are more or less not practicably formulatable in the vec model. Hence, the BEKK model is the most general one normally used.

To prove our results we employ  Markov chain theory combined with algebraic geometry to obtain the proper state spaces and properties like irreducibility. In particular, we
analyse stationarity and ergodicity for a general class of Markov chains which we call ``semi-polynomial'', because they generalise the polynomial Markov chains of
\cite{MOKPROP}, and then apply the general results to the special case of multivariate GARCH processes. Our approach in the present paper is most similar to the PhD
thesis of the first author \cite{BOUGARCH}, which, although it has never been published in an  accessible way (except for a summary of the main results without proofs in
\cite{Boussama2006}), has been relied upon in essential ways (see e.g. \cite{ComteLieberman2003}). Unfortunately, the statements and proofs in that thesis contain some
problematic issues and, hence,  details of the proofs and statements given in the present paper deviate in essential ways. One difference, for instance, is the use of (weak)
Feller chains, another regards the proper state spaces.   

The remainder of this paper is structured as follows. Below we briefly summarise some general notation. In Section \ref{sec:mgarch} we give a detailed definition of BEKK
GARCH models, their vec and vech parametrisations, state our main result on the stationarity and ergodicity of  multivariate GARCH processes and discuss its implications.
Thereafter, we define and analyse semi-polynomial Markov chains in Section \ref{sec:semipol}. Finally, we proof our main result for multivariate GARCH processes in
Section \ref{sec:mgarchproof}. A brief summary of some notions of algebraic geometry necessary to understand the statements of our main results on GARCH processes is
given in the appendix. There we have also collected some results from the theory of Markov chains which we are going to use in the proof of Theorem \ref{Satz 1.1.18}.
\subsection*{Notation}
For the natural numbers excluding zero we write $\mathbb{N}^*$.

We denote the set of real \(n\times d\) matrices by \(M_{n\times d}(\mathbb{R})\), the vector space of real \(d\times d\) matrices 
by \(M_d(\mathbb{R})\), the linear subspace of symmetric matrices by \(\mathbb{S}_d\), the positive semi-definite cone by \(\mathbb{S}_d^+\) 
and the (strictly) positive definite matrices by \(\mathbb{S}_d^{++}\). For a positive definite and positive semi-definite matrix \(A\in\mathbb{S}_d\) we
also write \(A>0\) and \(A\geq 0\), respectively. The transpose of a matrix \(A\in M_{n\times d}(\mathbb{R})\) will be denoted by \(A^t\) and the $d\times d$ identity
matrix by ${I}_d$. If the dimension is obvious from the context, we sometimes neglect the subscript.

Every matrix \(A\in M_{n\times d}(\mathbb{R})\) can be considered as a vector in \(\mathbb{R}^{nd}\) using the bijective \(\mathrm{vec}\)
transformation which stacks the columns of a matrix below one another beginning with the leftmost one. In the case of symmetric matrices, one often uses
the \(\mathrm{vech}\) transformation which maps \(\mathbb{S}_d\) bijectively to \(\mathbb{R}^{d(d+1)/2}\) by stacking the lower triangular portion of a matrix.
For instance, the matrix
\[\begin{pmatrix} a & b & c \\ b & d & e \\ c & e & f \end{pmatrix} \in \mathbb{S}_3\subseteq M_{3\times 3}(\bbr)\]
is mapped to the vector $(a,b,c,b,d,e,c,e,f)^t\in\bbr^9$ by the $\mathrm{vec}$ operator and to $(a,b,c,d,e,f)^t\in\bbr^6$ by the $\mathrm{vech}$ operator.
Finally, we denote for two matrices \(A\in M_{n\times d}(\mathbb{R})\) and \(B\in M_{r\times m}(\mathbb{R})\) the tensor
(Kronecker) product by \(A\otimes B\). 

For the relevant background on Markov chains and mixing we refer to any of the standard references, for instance, \cite{Doukhan1994, Meynetal1993}. 
%
%
\section{Stationarity and Geometric Ergodicity of BEKK Multivariate GARCH Models}
\label{sec:mgarch}
When one moves from a single-dimensional to a \(d\)-dimensional GARCH process, the univariate variance process becomes a \(d\times d\) covariance
matrix process \(\Sigma\). 
In the so-called {vec parametrisation} (see \cite{ENGMUL}) the general multivariate GARCH($p,q$) model with $p,q\in \mathbb{N}$ is given by
\begin{align}
 X_n &= \Sigma_n^{1/2} \epsilon_n, \label{3.3} \\
 \mathrm{vec}(\Sigma_n) &= \mathrm{vec}(C) + \sum\limits_{i=1}^q\tilde{A}_i\ \mathrm{vec}(X_{n-i}X_{n-i}^t) + 
  \sum\limits_{j=1}^p\tilde{B}_j\ \mathrm{vec}(\Sigma_{n-j}) \label{3.4}
\end{align}
for \(n\in\mathbb{N}^*\) where \((\epsilon_n)_{n\in\left.\mathbb{N}\right.^*}\) is  an \(\mathbb{R}^d\)-valued i.i.d.\ sequence and \(\Sigma_n^{1/2}\) denotes
the unique positive semi-definite square root of $\Sigma_n$ (i.e. the unique element \(\Sigma_n^{1/2}\) of $\mathbb{S}_d^+$ such that
$\Sigma_n=\Sigma_n^{1/2}\Sigma_n^{1/2})$. To ensure the positive semi-definiteness of the covariance matrix process \(\Sigma\) the initial values
\(\Sigma_0,\ldots,\Sigma_{1-p}\) and \(C\) have to be positive semi-definite and \(\tilde{A}_1,\ldots,\tilde{A}_q\) as well as \(\tilde{B}_1,\ldots,\tilde{B}_p\) need to 
be \(d^2\times d^2\) matrices mapping the vectorised positive semi-definite matrices into themselves. The initial values $X_0,X_{-1},\ldots, X_{1-q}$ may be arbitrary
elements of $\mathbb{R}^d$.

The restriction on the linear operators \(\tilde{A}_i\) and \(\tilde{B}_j\) necessary to ensure positive semi-definiteness gave rise to the so-called BEKK model (see again
\cite{ENGMUL}) which automatically ensures positive semi-definiteness:
\begin{align}
 X_n &= \Sigma_n^{1/2} \epsilon_n, \label{3.5} \\
 \Sigma_n &= C + \sum\limits_{i=1}^q\sum\limits_{k=1}^{l_i}\bar{A}_{i,k}X_{n-i}X_{n-i}^t\bar{A}_{i,k}^t + 
  \sum\limits_{j=1}^p\sum\limits_{r=1}^{s_j}\bar{B}_{j,r}\Sigma_{n-j}\bar{B}_{j,r}^t, \label{3.6}
\end{align}
where \(\bar{A}_{i,k}\) and \(\bar{B}_{j,r}\) are now arbitrary elements of \(M_d(\mathbb{R})\). \\
The BEKK model is equivalent to the \(\mathrm{vec}\) model with \(\tilde{A}_i=\sum_{k=1}^{l_i}\bar{A}_{i,k}\otimes\bar{A}_{i,k},\ i=1,\ldots,q\), and
\(\tilde{B}_j=\sum_{r=1}^{s_j}\bar{B}_{j,r}\otimes\bar{B}_{j,r},\ j=1,\ldots,p\). More details of the relations between \(\mathrm{vec}\) and BEKK GARCH models
are given in \cite{STEREL}.

If we take the symmetry of the matrices \(\Sigma_n\) into account, we can write the \(\mathrm{vec}\) (and thus the BEKK) model also in the 
\(\mathrm{vech}\) representation:
\begin{align}
 X_n &= \Sigma_n^{1/2} \epsilon_n, \label{3.7} \\
 \mathrm{vech}(\Sigma_n)   &= \mathrm{vech}(C) + \sum\limits_{i=1}^q A_i\ \mathrm{vech}(X_{n-i}X_{n-i}^t) + \sum\limits_{j=1}^p B_j\ 
  \mathrm{vech}(\Sigma_{n-j}), \label{3.8}
\end{align}
where \(A_i=H_d\tilde{A}_iK_d^t,\ B_j=H_d\tilde{B}_jK_d^t,\ i=1,\ldots,q,\ j=1,\ldots,p\) and the matrices \(H_d\) and \(K_d\) are the unique 
$H_d,K_d\in M_{\frac{d(d+1)}{2}\times d^2}(\mathbb{R})$ such that 
\[\mathrm{vech}(D)=H_d\ \mathrm{vec}(D),\ \mathrm{vec}(D) = K_d^t\ \mathrm{vech}(D)\quad\text{and}\quad H_dK_d^t={I}_{d(d+1)/2}\]
for every $D\in\mathbb{S}_d$, whose existence and uniqueness follows immediately, since both $\mathrm{vec}$ and $\mathrm{vech}$ are linear operators.
\begin{beispiel}
 Let us give a simple example for the GARCH equations when $d=2$. Consider the following BEKK model with $p=q=l_1=s_1=1$:
 \begin{equation}
  \Sigma_n = C + \begin{pmatrix} a & c \\ b & d \end{pmatrix}\,X_{n-1}X_{n-1}^t\,\begin{pmatrix} a & b \\ c & d \end{pmatrix} + 
   \begin{pmatrix} e & g \\ f & h \end{pmatrix}\,\Sigma_{n-1}\,\begin{pmatrix} e & f \\ g & h \end{pmatrix} \label{Equation 3.19}
 \end{equation}
 where $a,b,c,d,e,f,g,h$ are arbitrary real numbers. 

 From (\ref{Equation 3.19}) one immediately derives the corresponding $\mathrm{vec}$ model:
 \begin{equation}
  \mathrm{vec}(\Sigma_n)=\mathrm{vec}(C)+\begin{pmatrix}a^2 & ac & ca & c^2 \\ ab & ad & cb & cd \\ ba & bc & da & dc \\ b^2 & bd & db & d^2\end{pmatrix}\,
   \mathrm{vec}(X_{n-1}X_{n-1}^t) + \begin{pmatrix}e^2 & eg & ge & g^2 \\ ef & eh & gf & gh \\ fe & fg & he & hg \\ f^2 & fh & hf & h^2
   \end{pmatrix}\,\mathrm{vec}(\Sigma_{n-1}). \label{Equation 3.20}
 \end{equation}
 
 Since the matrices $H_2$ and $K_2$ are given by 
 \[H_2=\begin{pmatrix} 1 & 0 & 0 & 0 \\ 0 & 1 & 0 & 0 \\ 0 & 0 & 0 & 1 \end{pmatrix} \quad \text{and} \quad
  K_2=\begin{pmatrix} 1 & 0 & 0 & 0 \\ 0 & 1 & 1 & 0 \\ 0 & 0 & 0 & 1 \end{pmatrix},\]
 the associated $\mathrm{vech}$ parametrisation of (\ref{Equation 3.19}) and (\ref{Equation 3.20}) is 
 \[\mathrm{vech}(\Sigma_n) = \mathrm{vech}(C) + \begin{pmatrix} a^2 & 2ac & c^2 \\ ab & ad+ bc & cd \\ b^2 & 2bd & d^2 \end{pmatrix}\,
  \mathrm{vech}(X_{n-1}X_{n-1}^t) + \begin{pmatrix} e^2 & 2eg & g^2 \\ ef & eh + fg & gh \\ f^2 & 2fh & h^2 \end{pmatrix}\,\mathrm{vech}(\Sigma_{n-1}).\]
\end{beispiel}
\begin{definition}
 Let \((X_n)_{n\in\left.\mathbb{N}\right.^*}\)  be a \textnormal{GARCH}\((p,q)\) process in the BEKK representation satisfying 
 \begin{enumerate}
  \item
   \(C\) and the initial values \(\Sigma_{0},\ldots,\Sigma_{1-p}\) are positive definite,
  \item
   \((\epsilon_n)_{n\in\left.\mathbb{N}\right.^*}\) is an \(\mathbb{R}^d\)-valued i.i.d.\ sequence with distribution 
    \(\Gamma\), $\epsilon_1\in L^2$, \(\mathbb{E}[\epsilon_1]=0\) and \(\mathbb{E}[\epsilon_1\epsilon_1^t]={I}_d\),
  \item 
    \((\epsilon_n)_{n\in\left.\mathbb{N}\right.^*}\) is independent of 
     \(\mathcal{F}_{0}=\sigma(X_{1-q},X_{2-q},\ldots,X_0,\Sigma_{1-p},\Sigma_{2-p},\ldots, \Sigma_0)\).
 \end{enumerate}
 Then  \((X_n)_{n\in\left.\mathbb{N}\right.^*}\) is called \emph{standard \textup{GARCH}\((p,q)\) process}.
\end{definition}
$\mathbb{E}[\epsilon_1\epsilon_1^t]={I}_d$ is a simple normalisation making the volatility process identifiable and, hence, not really a restriction.
\begin{bemerkung}
 In \(X_n=\Sigma_n^{1/2}\epsilon_n\) one could also choose another transformation \(G(\Sigma_n)\) of the conditional covariance matrix \(\Sigma_n\) such that
 $G(\Sigma_n)^tG(\Sigma_n)=\Sigma_n$ instead of taking the square root. Boussama \cite{BOUGARCH} takes \(G(\Sigma_n)\) as a lower triangular matrix resulting 
 from the Cholesky decomposition. However, the exact form of this transformation does not have any impact on our results concerning stationarity and ergodicity
 provided that \(G\) is an appropriate ``smooth'' transformation such that the GARCH process fits into the setting of semi-polynomial Markov chains studied in 
 Section \ref {sec:semipol}.
 \label{Bemerkung 3.2.4}
\end{bemerkung}
Using the $\mathrm{vech}$ representation we embed a standard GARCH($p,q$) process into a Markov chain. Setting 
$\mathscr{C} := (\mathrm{vech}(C)^t,0,0,\ldots,0)^t\in\left(\mathbb{R}^{d(d+1)/2}\right)^p\times\left(\mathbb{R}^d\right)^q$ snd defining the process
\((Y_n)_{n\in\left.\mathbb{N}\right.^*}\) in \(\left(\mathbb{R}^{d(d+1)/2}\right)^p\times\left(\mathbb{R}^d\right)^q\) by
$Y_n:=\left(\mathrm{vech}(\Sigma_{n})^t,\mathrm{vech}(\Sigma_{n-1})^t,\ldots,\mathrm{vech}(\Sigma_{n-p+1})^t,X_n^t,X_{n-1}^t,\ldots,X_{n-q+1}^t\right)^t,$
one easily sees that
\[Y_n = \mathscr{C} + \begin{pmatrix}\sum\limits_{i=1}^q A_i\ \mathrm{vech}(X_{n-i}X_{n-i}^t) + \sum\limits_{j=1}^p B_j\ 
 \mathrm{vech}(\Sigma_{n-j})\\ \mathrm{vech}(\Sigma_{n-1})\\ \vdots\\ \mathrm{vech}(\Sigma_{n-p+1})\\ X_n\\ X_{n-1}\\ \vdots\\ X_{n-q+1}\end{pmatrix}.\]
From this one obtains immediately a regular (in the sense of  Definition \ref{Regulaere Abbildung}) map 
\[\varphi: \left(\left(\mathbb{R}^{d(d+1)/2}\right)^p\times\left(\mathbb{R}^d\right)^{q}\right) 
 \times\mathbb{R}^d\to\left(\mathbb{R}^{d(d+1)/2}\right)^p\times\left(\mathbb{R}^d\right)^{q}\mbox{ such that }Y_n = \varphi(Y_{n-1},X_n).\]
Next we need to define the set $W$ which the stationary standard GARCH($p,q$) process  takes its values in, as is to be seen. Set
\[\tilde{B} := \begin{pmatrix} B & 0 \\ 0 & 0 \end{pmatrix}\in M_{p\frac{d(d+1)}{2}+qd}(\mathbb{R})\,\mbox{ with }\,
 B=\begin{pmatrix}B_1    & B_2    & \ldots & B_{p-1} & B_p \\ I  & 0   & \ldots & 0   & 0  \\ 0    & I     & \ddots & \vdots  & 0  \\ 
 \vdots & \ddots & \ddots & 0   & \vdots \\ 0    & \ldots & 0      & I       & 0 \end{pmatrix}\in M_{p\frac{d(d+1)}{2}}(\mathbb{R})\]               
and consider the point $T$ satisfying 
\begin{equation}
 \label{eq:fixedpoint}
 T = \mathscr{C} + \tilde{B} T.  
\end{equation}
Existence and uniqueness of this point under the conditions of Theorem \ref{th:maintheo} is shown in Section \ref{Section 3.3.3}. 
Furthermore, we define $\varphi^n$ for $n\in \mathbb{N}^*$ recursively by $\varphi^1:=\varphi$ and 
\begin{align*}
 \varphi^{n+1}:\left(\left(\mathbb{R}^{d(d+1)/2}\right)^p\times\left(\mathbb{R}^d\right)^{q}\right)
  \times(\mathbb{R}^d)^{n+1}&\to\left(\mathbb{R}^{d(d+1)/2}\right)^p\times\left(\mathbb{R}^d\right)^{q}, \\ 
  \varphi^{n+1}(y, x_1,\ldots, x_n,x_{n+1})&=\varphi(\varphi^n(y,x_1,\ldots, x_n),x_{n+1})
\end{align*}
and set
\begin{align}
 W &= \raisebox{0.28cm}{\scriptsize{$Z$}}\overline{\bigcup\limits_{n\in\left.\mathbb{N}\right.^*}\left.
  \varphi\right.^n\left(T,\left(\mathbb{R}^d\right)^n\right)} \nonumber
\end{align}
with $\raisebox{0.18cm}{\scriptsize{$Z$}}\overline S$ denoting the closure of a set $S$ in the Zariski topology (see appendix). Since $\Sigma_n$ is  always positive
definite, we define
\[U := \underbrace{\mathrm{vech}(\mathbb{S}_d^{++})\times\ldots\times\mathrm{vech}(\mathbb{S}_d^{++})}_{p}\times
 \underbrace{\mathbb{R}^d\times\ldots\times\mathbb{R}^d}_{q}\]
and consider $W\cap U$ as the natural state space for $(Y_n)_{n\in\mathbb{N}^*}$. Finally,  $\mathcal{B}(W\cap U)$ denotes the Borel $\sigma$-algebra over 
$W\cap U$ inherited from the standard Borel $\sigma$-algebra on $\left(\mathbb{R}^{d(d+1)/2}\right)^p\times\left(\mathbb{R}^d\right)^{q}$, i.e.
$\mathcal{B}(W\cap U)$ is related to the usual Euclidean topology.
\begin{satz} 
 Let $(X_n)_{n\in\mathbb{N}^*}$ be a standard GARCH($p,q$) process.
 \begin{enumerate}
  \item
   If
    \begin{enumerate}
     \item[(H1)] 
      the distribution $\Gamma$ of  \(\epsilon_1\) is absolutely continuous with respect to the Lebesgue measure on \(\mathbb{R}^d\),
     \item[(H2)]
      the point zero is in the interior of \(E:=\mathrm{supp}({\Gamma})\) and 
     \item[(H3)] 
      the spectral radius of \(\sum_{i=1}^qA_i+\sum_{j=1}^pB_j\) is less than \(1\), 
    \end{enumerate}
   then the Markov chain \((Y_n)_{n\in\left.\mathbb{N}\right.^*}\) is positive Harris recurrent and geometrically ergodic on the state space 
   \((W\cap U,\mathcal{B}(W\cap U))\). 

   The strictly stationary solution \((X_n)_{n\in\mathbb{Z}}\) of the standard GARCH\((p,q)\) model associated with \((Y_n)_{n\in\mathbb{Z}}\) is unique and
   geometrically \(\beta\) - mixing. Furthermore, \(X_n\in L^2\) and \(\Sigma_n\in L^1\) for all \(n\in\mathbb{Z}\). The strictly stationary solution is, hence, also 
   weakly stationary.
  \item
   If there exists a weakly stationary solution for the standard GARCH\((p,q)\) model, then the spectral radius of the matrix
   \(\sum_{i=1}^qA_i+\sum_{j=1}^pB_j\) is less than \(1\).
 \end{enumerate}
 \label{Satz 3.3.9} \label{th:maintheo}
\end{satz}
\begin{bemerkung}
 \begin{enumerate}[(i)]
  \item
   If the initial values $\Sigma_0,\ldots, \Sigma_{1-p}$ are in $W$ and positive semi-definite but not positive definite, the geometric ergodicity still holds, since one easily
   sees that one then has $Y_p\in W\cap U$.
  \item
   Obviously, our conditions for strict stationarity imply also weak (or second order) stationarity. In the univariate case it is well-known that for a driving noise $\epsilon$
   with finite variance one may well choose the GARCH parameters in such a way that a unique stationary solution with infinite second moments exists (see 
   \cite{Basraketal2002,Bougeroletal1992b,Lindner2009}, for example). Extending the above result to cover such cases seems not possible at the moment, although in
   principle the main problem is ``only'' to find an appropriate function for the Foster-Lyapunov drift criterion. 
  \item
   Theorem \ref{Satz 3.3.9} should cover most of the BEKK multivariate GARCH models used in applications, since one usually wants a finite second moment and uses
   absolutely continuous noises $\epsilon$ (e.g. multivariate standard normal or standard $t_\nu$-distributed noises with $\nu>2$). Hence, (H3) will typically be the only
   condition that needs checking. 

   Note, however, that it is possible to weaken the assumptions (H1) and (H2) to the existence of a non-trivial absolutely continuous part of the innovation distribution with
   zero in the interior of its support (cf. \cite[Section 2.4]{Doukhan1994}). More precisely, if 
   \begin{enumerate}
    \item[(K1)]
     the distribution $\Gamma$ of  $\epsilon_1$ can be decomposed as $\Gamma=\Gamma_0+\Gamma_1$ with $\Gamma_0$ being non-trivial and absolutely 
     continuous with respect to the Lebesgue measure on $\mathbb{R}^d$ (cf., for instance, \cite[p. 174]{Sato1999} where the Lebesgue decomposition of any
     $\sigma$-finite measure on $\scrb(\bbr^d)$ is explained),
    \item[(K2)]
     the point zero is in the interior of $E=\mathrm{supp}(\Gamma_0)$
   \end{enumerate}
   and (H3) hold, then Theorem \ref{Satz 3.3.9} remains valid.

   We give a brief summary of the changes that have to be done in the following sections if one starts with (K1) and (K2) instead of (H1) and (H2):
   in the statements of Theorem \ref{Satz 2.2.1}, Proposition \ref{Proposition 2.2.3} and Theorem \ref{Satz 2.2.4} one simply has to replace $\Gamma$ by 
   $\Gamma_0$. Further the assumptions (A1) and (A2) in Section \ref{Section 2.3.1} have to be adapted in the same way as (H1) and (H2) have been weakend. 
   A crucial point is to replace the right hand side in Proposition \ref{Proposition 2.3.3} by the subprobability measure $F^k(T, \otimes_{i=1}^k\Gamma_0)$. 
   In the same manner we take $v$ in the proof of Proposition \ref{Proposition 2.3.4} equal to the subprobability measure $F^l(T, \otimes_{i=1}^l\Gamma_0)$. 
   The rest of the paper is not affected.
  \item
   If we completely omitted the assumptions (K1) and (K2) on the innovations, the mixing result might be no longer true (cf. \cite{Andrews1984}). However, one could 
   try to extend the idea of \cite{Doukhanetal2008} where the existence of a $\tau$-weakly dependent strictly stationary solution for a chain with infinite memory has 
   been shown under a Lipschitz-type condition, but the expression of this condition in terms of the matrices appearing in the BEKK representation seems to be very
   delicate. In some cases, like for instance the upcoming Example \ref{Beispiel 3.3.11}, the contraction condition is easily verified (in this case with the Orlicz function
   $\Phi(x)=x^2$). Hence, in such cases \cite[Theorem 3.1]{Doukhanetal2008} yields the existence of a $\tau$-weakly dependent strictly stationary solution of the
   standard GARCH model even if the innovation sequence does not possess an absolutely continuous component. 

   Finding an appropriate contraction condition in general appears highly non-trivial but may allow to show $\tau$-dependence without assuming (K1) and (K2) and 
   without using algebraic geometry. This would imply Central Limit Theorems (CLTs) and the validity of bootstrap procedures (see, for instance, \cite{Bickeletal1999}) 
   for the strictly stationary solution. However, for simulation purposes as well as CLTs and Strong Laws of Large Numbers when not starting with the stationary 
   distribution, geometric ergodicity and the ``right'' irreducible state space is very important (see again the upcoming Example \ref{Beispiel 3.3.11}). 
   The latter seem not to be obtainable under $\tau$-weak dependence conditions. Moreover, note also that $\tau$-weak dependence is a weaker notion than 
   strong mixing.
  \item
   Strong mixing conditions are, as $\tau$-dependence conditions, a way to derive limit theorems. Anyway, the CLT under strong mixing, even if the mixing 
   coefficients decline exponentially fast (which is the case for all geometrically ergodic Markov chains), needs a stronger order moment condition than the second order
   ones obtained in Theorem \ref{Satz 3.3.9} (see \cite{Haeggstroem2005} and also \cite{Herrndorf1983, Jones2004}). In \cite{Doukhanetal1994} 
   (see also \cite[Corollary 3]{Jones2004}) it has been shown that for a positive Harris recurrent and geometrically ergodic Markov chain $(X_t)_{t\in\bbz}$ on a state
   space $S$ with stationary distribution $\pi$, the CLT holds for any real-valued function $f$ defined on $S$ which satisfies $\int_S\pi(dx)f^2(x)\log^+|f(x)|<\infty$.
   In our case these references together with the obtained results in particular imply the CLT for functions $f$ such that 
   $\int_{W\cap U}\pi(dx)f^2(x)\log^+|f(x)|\leq C\cdot\int_{W\cap U}\pi(dx)V(x)$ (where $V$ is the function specified in the upcoming proof of Theorem 
   \ref{Satz 3.3.8}). Hence, for any $\varepsilon>0$, the CLT can be applied to $(1/2-\varepsilon)$th and $(1-\varepsilon)$th absolute powers of 
   $\Sigma_n$ and $X_n$, respectively. Moreover, combining our conditions on geometric ergodicity with the fourth order moment conditions of \cite{Hafner2003b}
   immediately gives sufficient conditions for the validity of more classical CLTs.
 \end{enumerate}
\end{bemerkung}
One might expect that the set $W\cap U$ spans the space $\left(\mathbb{R}^{d(d+1)/2}\right)^p\times\left(\mathbb{R}^d\right)^{q}$ and that, hence, 
the state space of a stationary GARCH process is ``non-degenerate''. However, this needs not to be true:
\begin{beispiel}
 Consider the following bivariate GARCH\((1,1)\) model:
 \begin{equation}
  \mathrm{vech}(\Sigma_n) = \mathrm{vech}(C) + A\ \mathrm{vech}(X_{n-1}X_{n-1}^t) + B\ \mathrm{vech}(\Sigma_{n-1}) \label{3.18}
 \end{equation}
 where \(A\) and \(B\) are two \(3\times 3\) matrices such that the spectral radius of \(A+B\) is less than \(1\) and \(BA = 0\). Such a GARCH model can be 
 obtained from a BEKK model with $l_1=s_1=1$ and $\bar B_{11}\bar A_{11}=0$. 

 Starting from the initial point \(T=(\mathrm{vech}(\Sigma_0)^t,X_0^t)^t\) given by equation \eqref{eq:fixedpoint}, we note that
 \(\mathrm{vech}(\Sigma_0) = \mathrm{vech}(C) + B\ \mathrm{vech}(\Sigma_0)\) and $X_0=0$ and obtain by iterating (\ref{3.18})
 \[\mathrm{vech}(\Sigma_n) = \mathrm{vech}(\Sigma_0) + A\ \mathrm{vech}(X_{n-1}X_{n-1}^t).\]
 Let \(f\) be the regular map from \(\mathbb{R}^4\) into \(\mathbb{R}^5\) given by
 \[(x_1,x_2,x_3,x_4)\mapsto f(x_1,x_2,x_3,x_4) := T + \begin{pmatrix} A\ (x_1^2, x_1x_2, x_2^2)^t \\ x_3 \\ x_4 \end{pmatrix}.\]
 Then \(W\) is the Zariski closure of the semi-algebraic set \(f(\mathbb{R}^4)\) (see \cite[Theorem 2.3.4]{BENALGSET}) and \(W\) has to be strictly contained in
 \(\mathbb{R}^5\) since \(\dim f(\mathbb{R}^4)\leq 4 = \dim\mathbb{R}^4\).
 \label{Beispiel 3.3.10}
\end{beispiel}
Note that the problem of degeneracy in this example lies in the non-invertibility of at least one of the two matrices $A,\,B$. Indeed, it is easy to see that $W$ is of 
full dimension and no such degeneracy as above can occur if $A,\,B$ (or $\bar A_{11},\,\bar B_{11}$ in the BEKK formulation) are both invertible and (H1), (H2) hold. 

Moreover one has to be very careful when using GARCH models not to use them outside the state space $W\cap U$. Typically one would simulate a stationary GARCH
process by starting with an arbitrary value and letting the process run. The values are only recorded after a burn-in period. The geometric ergodicity ensures that after an
appropriately long burn-in period the obtained values can be basically regarded as coming from the stationary dynamics. To ensure that this approach works, our results
show that one needs to start in $W\cap U$. One choice of the starting values always possible is $T$ which is easily calculated from the parameters by solving a system of
linear equations. Let us give an example where starting values outside $W\cap U$ indeed lead to a problem. 
\begin{beispiel}
 Consider the set-up of Example \ref{Beispiel 3.3.10} with $l_1=s_1=1$ and 
 \[\bar A_{11}=\begin{pmatrix} a & 0\\ 0 & 0\end{pmatrix},\,\, \bar B_{11}=\begin{pmatrix} 0 & 0\\ 0 & b\end{pmatrix},\]
 with $|a|<1,\,|b|<1$ being non-zero real numbers. Obviously (H3) is satisfied.

 It is then easy to see that the component corresponding to the second variance is constant in $W$, say it equals $\sigma_{22}$. If we start with an initial value
 $\Sigma_0$ with a  second variance $(\Sigma_0)_{22}$, then one sees easily  that $(\Sigma_n)_{22}=C_{22}+b^2(\Sigma_{n-1})_{22}$ for all 
 $n\in\mathbb{N}^*$. Obviously, this equation has a unique fixed point which must be equal to $\sigma_{22}$ and the right hand side corresponds to an injective map.
 By induction this implies $(\Sigma_n)_{22}\not =\sigma_{22}$ for all $n\in\mathbb{N}^*$ if the starting value satisfies $(\Sigma_0)_{22}\not =\sigma_{22}$.

 Hence, for such a starting value $Y_n \not\in W$ for all $n\in\mathbb{N}$ and thus the distribution of $Y_n$ can never converge in the total variation sense to the
 stationary distribution $\pi$. This means that we can never have geometric ergodicity when allowing such starting values outside $W$ but in $U$.
 \label{Beispiel 3.3.11}
\end{beispiel}
%
%
\section{Stationarity and Geometric Ergodicity of Semi-polynomial Markov Chains}
\label{sec:semipol}
In this section we consider a general class of Markov chains and prove criteria for stationarity and geometric ergodicity. We will apply the results later on to the 
special case of multivariate GARCH processes, but the general results of this section seem also of interest of their own, since they should be applicable to different 
models as well.

We consider Markov chains in $\mathbb{R}^n$ of the form  \(X_{t+1}=F(X_t, e_t)\) where \((e_t)_{t\in\mathbb{N}}\) is an $m$-dimensional i.i.d.\ sequence
and \(F\) is an appropriate map as follows. 

Let \(V\subseteq\mathbb{R}^n\) be an algebraic variety (cf.\ Definition \ref{Algebraische Varietaet}) and  \(U\)  an open subset of \(\mathbb{R}^n\) and let 
\(F: U\times\mathbb{R}^m\to U\) be a \(C^1\) - map such that there exist a \(C^1\) - map \(f: U\times\mathbb{R}^m\to\mathbb{R}^m\) and a map 
\(\varphi: \mathbb{R}^n\times\mathbb{R}^m\to\mathbb{R}^n\) satisfying:
\begin{enumerate}
 \item[(F1)]
  \(F(z,y)=\varphi(z,f(z,y))\) for all \((z,y)\in U\times\mathbb{R}^m\),
 \item[(F2)]
  \(\varphi(V\cap U\times\mathbb{R}^m)\subseteq V\cap U\),
 \item[(F3)]
  the map \((z,y)\mapsto\varphi(z,y)\) is regular in \((z,y)\) (cf.\ Definition \ref{Regulaere Abbildung}) and 
 \item[(F4)]
  for all \(z\in U\), the map \(f_z( \cdot )=f(z, \cdot )\) is a \(C^1\) - diffeomorphism from \(\mathbb{R}^m\) onto \(\mathbb{R}^m\) and the map
  \(U\times\mathbb{R}^m\to\mathbb{R}^m,\) \((z,y)\mapsto f_z^{-1}(y)\) is continuous in \((z,y)\) where \(f_z^{-1}( \cdot )\) denotes the inverse map 
  of \(f_z( \cdot )\).  
\end{enumerate}

The case when \(F\) is a regular map, i.e. when $f_z$ is the identity, has been considered in \cite{MOKPROP} under the name ``polynomial Markov chains''. 
Similarly to that paper we use extensively algebraic geometry (see the appendix for references and the most relevant definitions)  and drift criteria to show 
the stationarity and ergodicity of these Markov chains, but the presence of the additional diffeomorphism $f_z$ makes all proofs considerably more involved. 
Moreover, we always need to ensure that we stay in $U$.
\subsection{Properties of the Image Measure}
In a first step we consider how $F$ acts for a fixed first argument on the noise distribution, which will lead to $\psi$-irreducibility conditions in the next section.

In general, the image of \(\mathbb{R}^m\) under \(F_z( \cdot ):=F(z, \cdot )\) is a semi-algebraic set in \(\mathbb{R}^n\) with dimension less than \(n\) 
(see \cite[Theorem 2.3.4]{BENALGSET}). Thus the Lebesgue measure of this image is often zero.

Therefore we need to work with Hausdorff measures (see, for example, \cite{DIEELE} for a detailed introduction). We suppose that the algebraic variety \(V\) is 
equipped with a regular measure \(\mu_V\) defined on $(V,\mathcal{B}(V))$ where $\mathcal{B}(V)$ denotes again the Borel $\sigma$-algebra over $V$ inherited 
from the usual Euclidean topology. Recall that $\mu_V$ is said to be regular if, for any $A\in\mathcal{B}(V)$ and any $\delta>0$, there exist an open set
$U\in\mathcal{B}(V)$ and a compact set $K\in\mathcal{B}(V)$ such that $K\subseteq A\subseteq U$ and $\mu_V(U\backslash K)<\delta$. In the following we 
assume that this measure $\mu_V$ is obtained by equipping the regular set \(\mathcal{R}(V)\) of \(V\) (cf.\ Definition \ref{Regulaerer Punkt}) with an appropriate
Hausdorff measure which is extended by zero to the singular set \(\mathcal{S}(V)=V\backslash\mathcal{R}(V)\). Moreover we henceforth suppose that
\(\Gamma\) is a measure absolutely continuous with respect to the Lebesgue measure on \(\mathbb{R}^m\) with density \(\gamma\). For \(z\in V\cap U\) we 
denote by \(\Gamma_z\) the image measure of $\Gamma$ under $F_{z}$ in \(V\cap U\). Furthermore we define $E:=\mathrm{supp}(\Gamma)$ which is essentially 
also the domain of positivity of the density \(\gamma\). For the notion of smooth points we refer to \cite[A 20]{MOKPROP} or \cite[p. 42]{MUMALGGEO} and note that
the definition makes sense for general \(C^1\) - maps.
\begin{satz}
 Suppose that \(z_0\in V\cap U\) and \(F_{z_0}( \cdot )\) has a smooth point in \(\mathbb{R}^m\). Then \(\Gamma_{z_0}\) is absolutely continuous with respect to 
 the measure \(\mu_V\) and  has support \(F_{z_0}(E)\).
 \label{Satz 2.2.1}
\end{satz}
\begin{proof}
 First we denote by \(\Gamma_{z_0}^\prime\) the image measure \(f_{z_0}(\Gamma)\). One  obtains immediately by the Density Transformation Theorem that
 \(\Gamma_{z_0}^\prime\) is absolutely continuous with respect to the Lebesgue measure on \(\mathbb{R}^m\) with density
 \begin{equation}
  \gamma_{z_0}^\prime(y):=\gamma(f_{z_0}^{-1}(y))\cdot\frac{1}{\left|\det Df_{z_0}(f_{z_0}^{-1}(y))\right|}\ ,\quad y\in\mathbb{R}^m.
 \label{2.5}
 \end{equation}
 The support of $\Gamma^\prime_{z_0}$ is given by \(f_{z_0}(E).\) 

 There exists \(x_0\in\mathbb{R}^m\) such that \(F_{z_0}(x_0)=\varphi_{z_0}(y_0)\) is a regular point of \(V\) and \(rank(DF_{z_0}(x_0))=\dim V\) where
 \(y_0=f_{z_0}(x_0)\). Since \(F_{z_0}(\cdot)=\varphi_{z_0}(f_{z_0}(\cdot))\), we have \(DF_{z_0}(x_0)=D\varphi_{z_0}(y_0)\cdot Df_{z_0}(x_0)\). 
 Since \(f_{z_0}(\cdot)\) is a \(C^1\) - diffeomorphism, the linear map \(Df_{z_0}(x_0)\) is invertible. Thus \(rank(D\varphi_{z_0}(y_0))=\dim V\). 
 Since \(\varphi_{z_0}(y_0)\in\mathcal{R}(V)\)  and \(y_0\in\mathcal{R}(\mathbb{R}^m)\) (\(\mathbb{R}^m\) is a smooth algebraic variety,
 i.e.\ \(\mathcal{R}(\mathbb{R}^m)=\mathbb{R}^m\)), the regular map \(\varphi_{z_0}(\cdot)\) is smooth  at \(y_0\) and hence dominating (in the 
 sense of \cite[A23]{MOKPROP}). Applying \cite[Theorem 3.1]{MOKPROP} gives the result.
\end{proof}
\begin{proposition}
 Let \(z_0\in U\). Then, for every \(\epsilon>0\), there exists \(\alpha>0\) such that
 \[\left|\Gamma_z^\prime(B)-\Gamma_{z_0}^\prime(B)\right|=\left|f_z(\Gamma)(B)-f_{z_0}(\Gamma)(B)\right|<\epsilon\] 
 for all \(B\in\mathcal{B}(\mathbb{R}^m)\) and every \(z\in U\) with \(\left\|z-z_0\right\|<\alpha\).
 \label{Proposition 2.2.3}
\end{proposition}
\begin{proof}
 The image measure \(\Gamma_z^\prime=f_z(\Gamma)\) is absolutely continuous with respect to the Lebesgue measure on \(\mathbb{R}^m\) with density
 \(\gamma_z^\prime\) given by equation (\ref{2.5}). \\
 Let \(\epsilon>0\) and \(B\in\mathcal{B}(\mathbb{R}^m)\). The space of real-valued continuous functions on \(\mathbb{R}^m\) with compact support is
 dense in the \(L^1\) sense in the space of all Lebesgue-integrable functions on \(\mathbb{R}^m\). Thus, there exists a continuous function
 \(\tilde{\gamma}: \mathbb{R}^m\to\mathbb{R}\) with compact support \(K\) such that
 \begin{equation}
  \int_{\mathbb{R}^m}\left|\gamma(x)-\tilde{\gamma}(x)\right|\ dx < \frac{\epsilon}{3}.
 \label{2.6}
 \end{equation}
 Hence,
 \begin{align}
  \left|\Gamma_z^\prime(B)-\Gamma_{z_0}^\prime(B)\right| 
   &=\bigg|\int_B\gamma_z^\prime(y)\ dy - \int_B\gamma_{z_0}^\prime(y)\ dy\bigg| \nonumber \\
   & \leq\int_{\mathbb{R}^m}\left|\gamma(f_z^{-1}(y)) - \tilde{\gamma}(f_z^{-1}(y))\right|\cdot\left|\det Df_z^{-1}(y)\right|\ dy \nonumber \\
   &+\int_{\mathbb{R}^m}\Big|\tilde{\gamma}(f_z^{-1}(y))\left|\det Df_z^{-1}(y)\right| 
    - \tilde{\gamma}(f_{z_0}^{-1}(y))\left|\det Df_{z_0}^{-1}(y)\right|\Big|\ dy \nonumber \\
   &+\int_{\mathbb{R}^m}\left|\tilde{\gamma}(f_{z_0}^{-1}(y)) - \gamma(f_{z_0}^{-1}(y))\right|\cdot\left|\det Df_{z_0}^{-1}(y)\right|\ dy =:I_1+I_2+I_3.
    \nonumber
 \end{align}
 With (\ref{2.6}) we obtain immediately by substitution \(I_1<\epsilon/3\) and \(I_3<\epsilon/3\).

 \(\tilde{\gamma}\) is bounded on \(\mathbb{R}^m\) by \(\sup\tilde{\gamma}\) and  for all \(r>0\) such that
 \(\overline{B(z_0,r)}:=\left\{z\in\mathbb{R}^n: \left\|z-z_0\right\|\leq r\right\}\subseteq U\), the set
 \(C:=\left\{(z,y)\in U\times\mathbb{R}^m: \left\|z-z_0\right\|\leq r\text{  and  }f_z^{-1}(y)\in K\right\}\)
 is a compact set in \(U\times\mathbb{R}^m\), since the map \(\psi: U\times\mathbb{R}^m\to U\times\mathbb{R}^m,\ \psi(z,y):=(z,f_z(y))\) 
 is continuous and \(C=\psi(\overline{B(z_0,r)}\times K)\). Thus, there is a real number \(b>0\) such that, for all \((z,y)\in C\), 
 \(\big|\det Df_z^{-1}(y)\big|<b\). The map \(\tilde{\gamma}(f_z^{-1}(y))\cdot\big|\det Df_z^{-1}(y)\big|\) is hence bounded on \(C\) 
 by \(b\cdot\sup\tilde{\gamma}\). \\
 Let \(C_1\) be the projection of \(C\) on \(\mathbb{R}^m\) and suppose without loss of generality \(\left\|z-z_0\right\|\leq r\). Then, for all \(y\notin C_1\), we 
 have \(\tilde{\gamma}(f_z^{-1}(y))=\tilde{\gamma}(f_{z_0}^{-1}(y))=0\) which implies
 \[I_2=\int_{C_1}\Big|\tilde{\gamma}(f_z^{-1}(y))\left|\det Df_z^{-1}(y)\right| - \tilde{\gamma}(f_{z_0}^{-1}(y))\left|\det Df_{z_0}^{-1}(y)\right|\Big|\ dy.\]
 This integrand is dominated by \(2b\sup\tilde{\gamma}\) and converges pointwise to zero if \(z\) converges to \(z_0\) (cf. (F4)). Since \(b\) and 
 \(\sup\tilde{\gamma}\) are finite constants and \(C_1\) is compact the dominant \(2b\sup\tilde{\gamma}\) is integrable over \(C_1\). Hence, we 
 can apply the Dominated Convergence Theorem and get \(I_2\to 0\) as \(z\to z_0\), i.e.\ there exists \(0<\alpha<r\) such that \(I_2<\epsilon/3\) for all \(z\in U\) 
 with \(\left\|z-z_0\right\|<\alpha\).
\end{proof}
\begin{satz}
 Suppose that \(z_0\in V\cap U\) and \(F_{z_0}(\cdot)\) has a smooth point in \(\mathbb{R}^m\). Then 
 \begin{equation}
  \liminf\limits_{\substack{z\to z_0\\z\in V\cap U}}\Gamma_z(A)\geq\Gamma_{z_0}(A)
 \label{2.7}
 \end{equation}
 for every \(A\in\mathcal{B}(V\cap U)\) (the Borel $\sigma$-algebra inherited from the usual Euclidean topology).
 \label{Satz 2.2.4}
\end{satz}
\begin{proof}
 Let \(A\in\mathcal{B}(V\cap U)\) and \(\epsilon>0\). Since \(\varphi_{z_0}(\cdot)\) is dominating and \(\Gamma_{z_0}^\prime=f_{z_0}(\Gamma)\) is 
 absolutely continuous with respect to the Lebesgue measure on \(\mathbb{R}^m\) (cf.\ proof of Theorem \ref{Satz 2.2.1}), \cite[Theorem 3.2]{MOKPROP} yields a
 neighbourhood \(V_0\) of \(z_0\) in \(V\cap U\) such that 
 \(\varphi_z(\Gamma_{z_0}^\prime)(A)\geq\varphi_{z_0}(\Gamma_{z_0}^\prime)(A)-\frac{\epsilon}{2}\) for all \(z\in V_0\)
 which is equivalent to 
 \begin{equation}
  f_{z_0}(\Gamma)(\varphi_z^{-1}(A))\geq f_{z_0}(\Gamma)(\varphi_{z_0}^{-1}(A))-\frac{\epsilon}{2}\,\forall\, z\in V_0.
 \label{2.8}
 \end{equation}
 Due to Proposition \ref{Proposition 2.2.3}, there exists \(\alpha>0\) such that \(f_z(\Gamma)(B)\geq f_{z_0}(\Gamma)(B)-\frac{\epsilon}{2}\) for all 
 \(B\in\mathcal{B}(\mathbb{R}^m)\) and every \(z\in U\) with \(\left\|z-z_0\right\|<\alpha\). We choose \(B=\varphi_z^{-1}(A)\) and deduce for every 
 \(z\in U,\ \left\|z-z_0\right\|<\alpha,\)
 \begin{equation}
  f_z(\Gamma)(\varphi_z^{-1}(A))\geq f_{z_0}(\Gamma)(\varphi_z^{-1}(A))-\frac{\epsilon}{2}.
 \label{2.9}
 \end{equation}
 With (\ref{2.8}) and (\ref{2.9}) we obtain for all \(z\in V_0\) with  \(\left\|z-z_0\right\|<\alpha\) that 
 \(f_z(\Gamma)(\varphi_z^{-1}(A))\geq f_{z_0}(\Gamma)(\varphi_{z_0}^{-1}(A))-\epsilon.\) Since \(\varphi_z(\Gamma_z^\prime)=F_z(\Gamma)=\Gamma_z\), 
 this is equivalent to 
 \[\Gamma_z(A)\geq\Gamma_{z_0}(A)-\epsilon\qquad\forall z\in V_0\cap\left\{z\in\mathbb{R}^n: \left\|z-z_0\right\|<\alpha\right\}.\]
 This shows (\ref{2.7}) since \(\epsilon>0\) can be chosen arbitrarily small.
\end{proof}
\subsection{Stationarity and Ergodicity}
Culminating in Theorem \ref{Satz 2.3.5} we now gradually show Harris recurrence, geometric ergodicity and \(\beta\)-mixing for semi-polynomial Markov chains.
\subsubsection{Assumptions}
\label{Section 2.3.1}
Concerning the sequence \((e_t)_{t\in\mathbb{N}}\) we make the following additional assumptions for our semi-poly\-nomi\-al Markov chain:
\begin{enumerate}
 \item[(A1)]
  Every \(e_t\) has distribution \(\Gamma\) which is absolutely continuous with respect to the Lebesgue measure on \(\mathbb{R}^m\) with density \(\gamma\).
  Let \(E\) denote the support  of \(\Gamma\).
\end{enumerate}
We define for all \(k\in\mathbb{N}^*,\,k>1\) the functions \(F^k(z,y_1,\ldots,y_k):=F(F^{k-1}(z,y_1,\ldots,y_{k-1}),y_k)\) where 
\(z\in U,\ (y_1,\ldots,y_k)\in(\mathbb{R}^m)^k\) and set $F^1=F$.

With this notation we introduce for \(z\in V\cap U\) the orbit
\[S_z:=\bigcup\limits_{k\in\left.\mathbb{N}\right.^*}\left\{F^k(z,y_1,\ldots,y_k):\,y_1,\ldots,y_k\in E\right\}.\]
To prove the desired properties for semi-polynomial Markov chains we assume:
\begin{enumerate}
 \item[(A2)]
  There is a point \(a\in\mathrm{int}(E)\) and a point \(T\in V\cap U\) such that, for all \(z\in V\cap U\), the sequence \((X_t^z)_{t\in\mathbb{N}}\) defined by
  \(X_0^z=z\) and \(X_t^z=F(X_{t-1}^z,a)\) for \(t\geq 1\) converges to the point \(T\).
\end{enumerate}
T is called \emph{attracting point} of the chain \((X_t)_{t\in\mathbb{N}}\).

We set \(W:=\raisebox{0.15cm}{\scriptsize{$Z$}}\overline{S_T}\) the Zariski closure of the orbit \(S_T\). Note that obviously $T\in W$ and $W\subseteq V$.
To show uniqueness of the strictly stationary solution we need the assumption:
\begin{enumerate}
 \item[(A3)]
  Any strictly stationary solution of the Markov chain $X_{t+1}=F(X_t, e_t)$ takes its values in the algebraic variety \(W\cap U\).
\end{enumerate}
\begin{bemerkung}
 \begin{enumerate}
  \item
   If (A2) is satisfied, then \(T\) is a fixed point of \(F(\cdot, a)\), since \(F\) is continuous.
  \item
   It is obvious that \(W\) is an algebraic set since it is the Zariski closure of \(S_T\). In fact, it is even irreducible (cf. the upcoming Section \ref{Section 2.3.2}).
 \end{enumerate}
 \label{Bemerkung 2.3.1}
\end{bemerkung}

Strictly speaking \(W\cap U\) is not necessarily an algebraic variety, but, as it is the intersection of an algebraic variety in $\mathbb{R}^n$ and the set $U$ where
our Markovian dynamics are defined, we refer to it as an algebraic variety.
\subsubsection{Algebraic Variety of States}
\label{Section 2.3.2}
In this subsection we suppose that the assumptions (A1) and (A2) hold. We will show that \(W\), defined as above, is indeed an algebraic variety
which we will call the Markov chain's \emph{algebraic variety of states}.

Let \((D_k)_{k\in\left.\mathbb{N}\right.^*}\) be the sequence of subsets of \(U\) defined by \(D_k:=F^k(T,E^k).\) Since \(F(T,a)=T\) (cf.\ Remark 
\ref{Bemerkung 2.3.1} (i)) we obtain \(D_k=F^k(F(T,a),E^k)=F^{k+1}(T,{\left\{a\right\}\times E^k})\subseteq D_{k+1},\) i.e.\ the sequence 
\((D_k)_{k\in\left.\mathbb{N}\right.^*}\) is an ascending sequence of subsets of \(\mathbb{R}^n\).

We set \(W_k:=\raisebox{0.15cm}{\scriptsize{$Z$}}\overline{F^k(T,(\mathbb{R}^m)^k)}\). Then we have 
\(W_k=\raisebox{0.15cm}{\scriptsize{$Z$}}\overline{\varphi^k(T,(\mathbb{R}^m)^k)}\) (defining \(\varphi^k\) analogously to \(F^k\)) since \(f_T(\cdot)\) is a 
\(C^1\) - diffeomorphism.
\begin{lemma}
 For all \(k\in\mathbb{N}^*\) we have \(W_k=\raisebox{0.14cm}{\scriptsize{$Z$}}\overline{D_k}\).
 \label{Lemma 2.3.7}
\end{lemma}
\begin{proof}
 To this end consider the map \(f_T^{(k)}: (\mathbb{R}^m)^k\to(\mathbb{R}^m)^k,\) \((y_1,\ldots,y_k)\mapsto(x_1,\ldots,x_k)\) where 
 \(x_1=f_T(y_1),\ x_2=f_{F(T,y_1)}(y_2),\ldots,\ x_k=f_{F^{k-1}(T,y_1,\ldots,y_{k-1})}(y_k)\). \\
 Due to the properties of \(f\) and \(F\) (in particular (F4)), it is clear that \(f_T^{(k)}\) is bijective, continuous and its inverse is continuous as well, 
 i.e.\ \(f_T^{(k)}\) is a homeomorphism. \\
 Assumption (A2) implies that \(E^k\) contains an open ball of \((\mathbb{R}^m)^k\). Thus, since \(f_T^{(k)}\) is homeomorphic, \(f_T^{(k)}(E^k)\) contains 
 an open ball of \((\mathbb{R}^m)^k\).

 From \cite[Corollary  3.4.5]{BENALGSET} we obtain \(\raisebox{0.26cm}{\scriptsize{$Z$}}\overline{f_T^{(k)}(E^k)}=(\mathbb{R}^m)^k\). This shows 
 \[W_k=\raisebox{0.15cm}{\scriptsize{$Z$}}\overline{F^k(T,(\mathbb{R}^m)^k)}=\raisebox{0.15cm}{\scriptsize{$Z$}}
  \overline{\varphi^k(T,(\mathbb{R}^m)^k)}=\raisebox{0.35cm}{\scriptsize{$Z$}}
  \overline{\varphi^k\left(T,\raisebox{0.26cm}{\scriptsize{$Z$}}\overline{f_T^{(k)}(E^k)}\right)}.\]
 Since \(\varphi^k(T,\cdot)\) is regular (cf.\ (F3)), \(\varphi^k(T,\cdot)\) is continuous with respect to the Zariski topology due to Proposition 
 \ref{Proposition 1.2.17}. Hence, 
 \(W_k=\raisebox{0.32cm}{\scriptsize{$Z$}}\overline{\varphi^k\left(T,f_T^{(k)}(E^k)\right)}=\raisebox{0.15cm}{\scriptsize{$Z$}}\overline{F^k(T,E^k)} = 
  \raisebox{0.14cm}{\scriptsize{$Z$}}\overline{D_k}\)
 which proves \(W_k\) to be the Zariski closure of \(D_k\).
\end{proof}
\begin{lemma}
 \(W_k\) is irreducible for all \(k\in\mathbb{N}^*\).
 \label{Lemma 2.3.8}
\end{lemma}
\begin{proof}
 If we suppose that there is \(k\in\mathbb{N}^*\) such that \(W_k=V_1\cup V_2\) where \(V_1\) and \(V_2\) are algebraic sets with \(V_1\subsetneq W_k\) and
 \(V_2\subsetneq W_k\), then 
 \[(\mathbb{R}^m)^k=\left(\varphi_T^k\right)^{-1}(W_k)=\underbrace{\left(\varphi_T^k\right)^{-1}(V_1)}_{(*)}\cup
  \underbrace{\left(\varphi_T^k\right)^{-1}(V_2)}_{(**)}\]
 where \(\varphi_T^k(\cdot)=\varphi^k(T,\cdot)\). Now \((*)\) and \((**)\) are algebraic sets, because \(V_1\) and \(V_2\) are algebraic sets 
 and \(\varphi_T^k(\cdot)\) is continuous with respect to the Zariski topology (see proof of Lemma \ref{Lemma 2.3.7}). Since 
 \(\left(\varphi_T^k\right)^{-1}(V_i)\subsetneq(\mathbb{R}^m)^k\) for \(i=1,2\) (otherwise 
 \(W_k=\raisebox{0.15cm}{\scriptsize{$Z$}}\overline{\varphi^k(T,(\mathbb{R}^m)^k)}\subseteq\raisebox{0.15cm}{\scriptsize{$Z$}}\overline{V_i}=V_i\) 
 which would be a contradiction to \(V_i\subsetneq W_k\)), this would prove \((\mathbb{R}^m)^k\) to be reducible which is a contradiction.
\end{proof}
\begin{proposition}
 There exists \(l\in\mathbb{N}^*\) such that \(W_k=W_l\) for all \(k\geq l\) and \(W=W_l\). In particular, \(W\) is an algebraic variety.
 \label{Proposition 2.3.9}
\end{proposition}
\begin{proof}
 From \cite[Corollary  3.4.5]{BENALGSET} it follows that  if \(V_1\subseteq V_2\subseteq V_3\subseteq\ldots\) is an ascending sequence of algebraic varieties in
 \(\mathbb{R}^n\) then there exists \(l\in\mathbb{N}^*\) such that \(V_k = V_l\) for all \(k\geq l\).

 Lemma \ref{Lemma 2.3.7} and Lemma \ref{Lemma 2.3.8} show that \((W_k)_{k\in\left.\mathbb{N}\right.^*}\) is an ascending sequence of algebraic varieties 
 and so there exists \(l\in\mathbb{N}^*\) such that \(W_k=W_l\) for all \(k\geq l\). We then observe that
 \[S_T=\bigcup\limits_{k\in\left.\mathbb{N}\right.^*}\underbrace{\left\{F^k(T,y_1,\ldots,y_k):\,y_1,\ldots,y_k\in E\right\}}_{=F^k(T,E^k)}=
  \bigcup\limits_{k\in\left.\mathbb{N}\right.^*}D_k.\]
 Since 
 \[\raisebox{0.28cm}{\scriptsize{$Z$}}\overline{\bigcup\limits_{k\in\left.\mathbb{N}\right.^*}D_k}\subseteq
  \raisebox{0.28cm}{\scriptsize{$Z$}}\overline{\underbrace{\bigcup\limits_{k\in\left.\mathbb{N}\right.^*}\raisebox{0.15cm}
  {\scriptsize{$Z$}}\overline{D_k}}_{=\bigcup\limits_{k\in\left.\mathbb{N}\right.^*}W_k=W_l}}=\raisebox{0.15cm}{\scriptsize{$Z$}}\overline{W_l}=W_l\]
 and \(W_l=\raisebox{0.15cm}{\scriptsize{$Z$}}\overline{D_l}\subseteq\raisebox{0.15cm}{\scriptsize{$Z$}}\overline{\cup_{k\in\left.\mathbb{N}\right.^*}D_k}\),
 we obtain \(W=\raisebox{0.15cm}{\scriptsize{$Z$}}\overline{S_T}=W_l\).
\end{proof}
\begin{lemma}
 For all \(k\in\mathbb{N}^*\) we have \(F^k(W\cap U,(\mathbb{R}^m)^k)\subseteq W\cap U\). Hence, the Markov chain can be restricted to the variety of states 
 \(W\cap U\).
 \label{Lemma 2.3.10}
\end{lemma}
\begin{proof}
 With the definition of the subsets \(D_k\) and \(W_k\), respectively, one has for all $ k\in\mathbb{N}^*$ 
 \begin{equation}
  \varphi(D_k,\mathbb{R}^m)=F(D_k,\mathbb{R}^m)=F(\underbrace{F^k(T,E^k)}_{\subseteq F^k(T,(\mathbb{R}^m)^k)},\mathbb{R}^m)
  \subseteq F^{k+1}(T,(\mathbb{R}^m)^{k+1})\subseteq W_{k+1}\cap U\subseteq W\cap U .
 \label{2.11}
 \end{equation}
 The continuity of regular maps with respect to the Zariski topology and the regularity of \(\varphi\) yield
 \begin{align}
  F(W\cap U,\mathbb{R}^m) &= \varphi(W\cap U,\mathbb{R}^m)\subseteq \varphi(W,\mathbb{R}^m)=\varphi\left(\raisebox{0.15cm}{\scriptsize{$Z$}}
   \overline{D_l},\mathbb{R}^m\right)\nonumber\subseteq\raisebox{0.23cm}{\scriptsize{$Z$}}\overline{\varphi\left(\raisebox{0.15cm}{\scriptsize{$Z$}}
   \overline{D_l},\mathbb{R}^m\right)}=\raisebox{0.15cm}{\scriptsize{$Z$}}\overline{\varphi(D_l,\mathbb{R}^m)}\stackrel{\eqref{2.11}}{\subseteq}W. \nonumber
 \end{align}
 Since we assume $F:U\times \mathbb{R}^m\to U$, we have $F(W\cap U,\mathbb{R}^m)\subseteq  W\cap U$. By induction we have 
 \(F^k(W\cap U,(\mathbb{R}^m)^k)\subseteq W\cap U\) for all \(k\in\mathbb{N}^*\). Hence we can restrict the Markov chain to the variety of states \(W\cap U\).
\end{proof}
\begin{proposition}
 For all \(A\in\mathcal{B}(W\cap U)\) and all \(k\geq l\) 
 \[\liminf\limits_{\substack{z\to T\\z\in W\cap U}}P^k(z,A)\geq P^k(T,A),\]
 where \(P^k\) is the \(k\)-step transition probability kernel of the Markov chain \((X_t)_{t\in\mathbb{N}}\).
 \label{Proposition 2.3.3}
\end{proposition}
\begin{proof}
 Since \(\varphi_T^k(\cdot)=\varphi^k(T,\cdot)\) is regular and dominating for all \(k\geq l\) (since 
 \(\raisebox{0.15cm}{\scriptsize{$Z$}}\overline{\varphi^k(T,(\mathbb{R}^m)^k)}=W_k=W\) for all \(k\geq l\), cf.\ Proposition \ref{Proposition 2.3.9}), 
 \cite[A 23]{MOKPROP} implies that \(\varphi_T^k( \cdot )\) has a smooth point.

 Similar to the proof of Theorem \ref{Satz 2.2.1} we can now show that the map \(F^k(T, \cdot )\) has a smooth point in \((\mathbb{R}^m)^k\). Let
 \(x_0\in(\mathbb{R}^m)^k\) be the smooth point of \(\varphi_T^k( \cdot )\), i.e.\ \(\varphi^k(T,x_0)\in\mathcal{R}(W)\) and 
 \(rank\left(D\varphi_T^k(x_0)\right)=\dim W\). Then \(F^k\Big(T,\left(f_T^{(k)}\right)^{-1}(x_0)\Big)=\varphi^k(T,x_0)\) and 
 \[DF^k(T, \cdot )\Big(\left(f_T^{(k)}\right)^{-1}(x_0)\Big)=D\varphi_T^k(x_0)\cdot Df_T^{(k)}\Big(\left(f_T^{(k)}\right)^{-1}(x_0)\Big)\]
 where the linear map \(Df_T^{(k)}(x)\) is invertible for all \(x\in(\mathbb{R}^m)^k\) (to this end note that \(f_T^{(k)}\) is not only continuous
 but also differentiable and that \(Df_T^{(k)}\) is a block matrix with lower triangle structure where the blocks on the diagonal are invertible). Hence the matrix on 
 the left hand side has also rank \(\dim W\) and \(\big(f_T^{(k)}\big)^{-1}(x_0)\) is a smooth point of \(F^k(T, \cdot )\). \\
 Finally, note that \(P^k(z,A)=F^k(z,\otimes_{i=1}^k\Gamma)(A)\) where \(\Gamma\) is the distribution of every \(e_t\) (cf.\ (A1)) and we conclude with Theorem
 \ref{Satz 2.2.4}.
\end{proof}
\subsubsection[Harris Recurrence, Ergodicity and \texorpdfstring{$\beta$}{beta} - Mixing]{Harris Recurrence, Ergodicity and $\beta$ - Mixing}
In this subsection we will prove the promised properties of semi-polynomial Markov chains under  a Foster-Lyapunov-condition. First we show irreducibility 
on the algebraic variety of states and aperiodicity. As usual, $\psi$ denotes a maximal irreducibility measure. 
\begin{proposition}
 Suppose that (A1) and (A2) hold. Then the semi-polynomial Markov chain \((X_t)_{t\in\mathbb{N}}\) is \(\psi\) - irreducible and aperiodic on the state space 
 \((W\cap U, \mathcal{B}(W\cap U))\). \\
 Moreover, the support of \(\psi\) has non-empty interior.
 \label{Proposition 2.3.4}
\end{proposition}
\begin{proof}
 (1) Due to Proposition \ref{Proposition 2.3.3} we have for all \(A\in\mathcal{B}(W\cap U)\)
 \begin{equation}
  \liminf\limits_{\substack{z\to T\\z\in W\cap U}}P^l(z,A)\geq P^l(T,A).
 \label{2.12}	 
 \end{equation}
 We define a probability measure \(\nu\) on the state space \((W\cap U,\mathcal{B}(W\cap U))\) by \(\nu(A):=P^l(T,A),\,A\in\mathcal{B}(W\cap U).\)
 Then, for every \(A\in\mathcal{B}(W\cap U)\) with \(\nu(A)\neq 0\), there exists due to (\ref{2.12}) a neighbourhood \(W_1\) of \(T\) in \(W\cap U\) such that
 \begin{equation}
  P^l(z,A)\geq\frac{\nu(A)}{2}\,\forall\, z\in W_1.
 \label{2.13}
 \end{equation}
 (2) Let \(K=\left\{z_1,\ldots,z_r\right\}\subseteq W\cap U\) for some \(r\in\mathbb{N}^*\). We are going to show that there is a \(q\in\mathbb{N}^*\) such that 
 \(P^q(z_i,W_1)>0\) \(\forall i\in\left\{1,\ldots,r\right\}.\) To this end, consider for \(i=1,\ldots,r\) the sequences \((X_t^{z_i})_{t\in\mathbb{N}}\) defined by 
 \[X_0^{z_i}=z_i\qquad\text{ and }\qquad X_t^{z_i}=F(X_{t-1}^{z_i},a),\ t\geq 1\]
 where \(a\in\mathrm{int}(E)\) as in (A2).

 Due to assumption (A2) there is \(q\in\mathbb{N}^*\) such that \(X_q^{z_i}=F^q(z_i,a,\ldots,a)\in W_1\) \(\forall i\in\left\{1,\ldots,r\right\}\). Since 
 \(F^q: W\cap U\times(\mathbb{R}^m)^q\to W\cap U\) is continuous, there exists for every \(i\in\left\{1,\ldots,r\right\}\) a neighbourhood \(U_i\) of 
 \((z_i,a,\ldots,a)\) in \(W\cap U\times(\mathbb{R}^m)^q\) such that
 \[ F^q(y,y_1,\ldots,y_q)\in W_1\,\forall\,(y,y_1,\ldots,y_q)\in U_i.\]
 Then, for all \(i\in\left\{1,\ldots,r\right\}\), \(U_i\) contains \(U_i^\prime\times U_{(a,\ldots,a)}^i\) where \(U_i^\prime\) and
 \(U_{(a,\ldots,a)}^i\) are suitable neighbourhoods of \(z_i\) in \(W\cap U\) and \((a,\ldots,a)\) in \((\mathbb{R}^m)^q\), respectively.

 We define \(U_{(a,\ldots,a)}:=\bigcap\limits_{i=1}^r U_{(a,\ldots,a)}^i\) which is clearly also a neighbourhood of \((a,\ldots,a)\) in \((\mathbb{R}^m)^q\).
 Then we have \(F^q(z_i,y_1,\ldots,\allowbreak y_q)\in W_1\) for all \( i\in\left\{1,\ldots,r\right\}\) and \((y_1,\ldots,y_q)\in U_{(a,\ldots,a)}\).
 Since \(U_{(a,\ldots,a)}\) contains itself \(U_a\times\ldots\times U_a\) where \(U_a\) is an appropriate neighbourhood of \(a\) in \(\mathbb{R}^m\), we deduce 
 for all \(i\in\left\{1,\ldots,r\right\}\):
 \begin{align}
  P^q(z_i,W_1) &\geq\mathbb{P}((e_1,\ldots,e_q)\in U_{(a,\ldots,a)})\geq\mathbb{P}(e_1\in U_a)^q = \Gamma(U_a)^q.
 \label{2.14}
 \end{align}
 (3) Let \(A\in\mathcal{B}(W\cap U)\) with \(\nu(A)\neq 0\). As in (1), \(W_1\) denotes the neighbourhood of \(T\) in \(W\cap U\) such that \(P^l(z,A)\geq\nu(A)/2\) 
 for all \(z\in W_1\). Using the Chapman-Kolmogorov equation (cf.\ \cite{ Meynetal1993} Theorem 3.4.2), we obtain for every \(i\in\left\{1,\ldots,r\right\}\)
 \begin{align}
  P^{q+l}(z_i,A) &=\int_{W\cap U} P^q(z_i,dy) P^l(y,A)\geq\int_{W_1} P^q(z_i,dy) P^l(y,A)
   \stackrel{\eqref{2.13}}{\geq}\frac{\nu(A)}{2}\cdot\int_{W_1} P^q(z_i,dy) \nonumber \\
   & = \frac{\nu(A)}{2}\cdot P^q(z_i,W_1)\stackrel{\eqref{2.14}}{\geq}\frac{\Gamma(U_a)^q}{2}\cdot\nu(A). \nonumber
 \end{align}
 Due to assumption (A2) \(U_a\) contains an open set of \(E\). Thus \(\Gamma(U_a)>0\). This implies that the chain \((X_t)_{t\in\mathbb{N}}\) is \(\nu\) - irreducible
 (and thus also \(\psi\) - irreducible due to \cite[Proposition 4.2.2]{Meynetal1993}). \\
 (4) To show aperiodicity, we suppose the chain to be periodic with period \(d\). Due to \cite[Theorem 5.4.4]{Meynetal1993} there exist disjoint sets
 \(D_1,\ldots,D_d\in\mathcal{B}(W\cap U)\) such that 
 \[\mbox{ (i) } \, P(z,D_{(i\mod d) + 1}) = 1\,\forall i=1,\ldots,d \mbox{ and } z\in D_i\, \quad\mbox{ and }\quad\mbox{ (ii) }\, \psi((\cup_{i=1}^d D_i)^c) = 0.\]

 Since \(\psi\succ\nu\) (cf. \cite[Proposition 4.2.2]{Meynetal1993}), \((\cup_{i=1}^d D_i)^c\) is also a \(\nu\) - null set. Obviously there must be a set \(D_i\) with
 positive \(\nu\) - measure, let this set be \(D_1\) without loss of generality. \\
 Let \(x\in D_1\) and \(y\in D_d\). For \(K:=\left\{x,y\right\}\) we have just shown in step (3) that
 \[P^{q+l}(x,D_1)>0\qquad\text{and}\qquad P^{q+l}(y,D_1)>0\]
 for some \(q\in\mathbb{N}^*\). Hence, the integers \(q+l\) and \(q+l-1\) are divisible by \(d\). Consequently \(d=1\). \\
 (5) We have shown in step (3) that \((X_t)_{t\in\mathbb{N}}\) is \(P^l(T, \cdot )\) - irreducible. Since \(F^l(T, \cdot )\) has a smooth point in 
 \((\mathbb{R}^m)^l\) (cf.\ proof of Proposition \ref{Proposition 2.3.3}) Theorem \ref{Satz 2.2.1} implies that \(P^l(T, \cdot )=F^l(T,\otimes_{i=1}^l\Gamma)\) 
 is absolutely continuous with respect to the measure \(\mu_W\). Hence, \(\mathrm{int}(\mathrm{supp}\ P^l(T, \cdot ))\neq\emptyset\) and we also obtain that
 \(\mathrm{int}(\mathrm{supp}\ \psi)\neq\emptyset\).
\end{proof}
We can now state our main result for semi-polynomial Markov chains. Therefore we use the standard notation $PV(x):=\bbE[V(X_1)|X_0=x]=\bbE_x[V(X_1)]$.
\begin{satz}
 Suppose (A1) and (A2) are valid. If in addition the Foster-Lyapunov-condition holds, i.e. there exist a small set \(C\in \mathcal{B}(W\cap U)\), positive constants
 \(\alpha<1,\ b<\infty\) and a function \(V\geq 1\) such that
 \begin{equation}
  PV(x)\leq\alpha\cdot V(x) + b\cdot{1}_C(x) \forall x\in W\cap U
 \label{Foster-Lyapounov}
 \tag{\text{FL}},
 \end{equation}
 then the semi-polynomial  Markov chain \((X_t)_{t\in\mathbb{N}}\) is positive Harris recurrent and geometrically ergodic on the algebraic variety of states 
 \(W\cap U\). Furthermore, the strictly stationary process \((X_t)_{t\in\mathbb{Z}}\) is geometrically \(\beta\) - mixing and \(\pi(V):=\bbE[V(X_t)]<\infty\).
 \label{Satz 2.3.5}
\end{satz}
\begin{proof}
 Due to Proposition \ref{Proposition 2.3.4} and the assumptions (A1) and (A2), \((X_t)_{t\in\mathbb{N}}\) is \(\psi\) - irreducible and aperiodic on 
 \((W\cap U,\mathcal{B}(W\cap U))\). We conclude by using the following Theorem \ref{Satz 1.1.18}.
\end{proof}
\begin{satz}
 Let \((X_t)_{t\in\mathbb{N}}\) be a \(\psi\) - irreducible Markov chain on  a state space \((S, \mathcal{B}(S))\) with transition probability kernel \(P\).
 If the chain is aperiodic and the Foster-Lyapunov-condition holds, i.e. there exist a small set \(C\in\mathcal{B}(S)\), positive constants \(\alpha<1,\ b<\infty\) and a
 function \(V\geq 1\) such that
 \begin{equation*}
  PV(x)\leq\alpha\cdot V(x) + b\cdot{1}_C(x)\,\,\forall x\in S, 
 \end{equation*} 
 then \((X_t)_{t\in\mathbb{N}}\) is positive Harris recurrent, geometrically ergodic and the strictly stationary process \((X_t)_{t\in\mathbb{Z}}\) is geometrically
 \(\beta\) - mixing. Furthermore, \(\pi(V)<\infty\).
 \label{Satz 1.1.18}
\end{satz}
\begin{proof}
 Since the Foster-Lyapunov-condition holds, the non-negative functions $V^\prime:=V-1,\,f:=1-\al$ and $s:=b\mathds{1}_C$ satisfy the assumption of Theorem
 \ref{Theorem B.1}. Hence we obtain for the first return time to $C$, denoted by $\tau_C$, 
 \[\bbE_x[\tau_C]=\frac{1}{1-\al}\bbE_x\bigg[\sum_{k=0}^{\tau_C-1} f(X_k)\bigg]\leq\frac{1}{1-\al}\bigg(V(x)-1+
  \underbrace{\bbE_x\bigg[\sum_{k=0}^{\tau_C-1} s(X_k)\bigg]}_{=b\mathds{1}_C(x)}\bigg)<\infty\,\forall x\in S\]
 and thus obviously $L(x,C)=\bbP_x(\tau_C<\infty)=1$ for all $x\in S$. Since every small set is also petite (cf. \cite{MeynTweedie1992} or \cite{Meynetal1993}),
 Proposition \ref{Proposition B.2} yields Harris recurrence of $(X_t)_{t\in\bbn}$. Again the Foster-Lyapunov-condition shows that $V^\prime:=(1-\al)^{-1}V,\,f:=V$ 
 and $b^\prime:=(1-\al)^{-1}b$ satisfy (ii) of Theorem \ref{Theorem B.3} which implies that $(X_t)_{t\in\bbn}$ is positive and $\pi(V)<\infty$. It is once more 
 the same condition that yields directly geometric ergodicity by virtue of Theorem \ref{Theorem B.4}. Finally, combining \cite[Proposition 1 (1)]{Davydov1973} with
 (\ref{Equation B.1}) and $\pi(V)<\infty$, we deduce that the strictly stationary process $(X_t)_{t\in\bbz}$ is geometrically $\beta$-mixing.
\end{proof}
\begin{satz}
 Suppose the setting of Theorem \ref{Satz 2.3.5} and assume in addition that (A3) holds. Then the strictly stationary process is unique.
 \label{Satz 2.3.6}
\end{satz}
\begin{proof}
 If there is another \(P\) - invariant probability measure \(\pi^\prime\), then \(\mathrm{supp}(\pi^\prime)\subseteq W\cap U\) due to (A3). Since the chain
 \((X_t)_{t\in\mathbb{N}}\) is recurrent on \(W\cap U\) (Theorem \ref{Satz 2.3.5}), it has at most one \(P\) - invariant probability measure on 
 \((W\cap U,\mathcal{B}(W\cap U))\) (cf.\ \cite[Theorem 10.4.4]{ Meynetal1993}). Therefore the strictly stationary solution is unique.
\end{proof}
\begin{bemerkung}
 Note that the whole theory from algebraic geometry has only been used to prove irreducibility and aperiodicity on the algebraic variety of states $W\cap U$. The 
 results thereafter (Theorems \ref{Satz 1.1.18} and \ref{Satz 2.3.6}) are consequences of the theory of Markov chains.
\end{bemerkung}
%
%
\section[Proof of the Main Theorem]{Proof of the Main Theorem \ref{th:maintheo}}
\label{sec:mgarchproof}
In this section we gradually prove our main result for multivariate GARCH processes, Theorem \ref{th:maintheo}.
\subsection{GARCH Processes as Semi-Polynomial Markov Processes}
\label{sec:mgarchsemipol}
First we show that the autoregressive representation of standard GARCH processes involving the function $\varphi$ leads to a semi-polynomial Markov chain.
\begin{lemma}
 The mapping $G:\mathbb{S}_d^{+}\to\mathbb{S}_d^{+}$  which maps \(\Sigma\) to \(\Sigma^{1/2}\) is a \(C^1\) - diffeomorphism on \(\mathbb{S}_d^{++}\).
 \label{Lemma 3.2.5}
\end{lemma}
\begin{proof}
 We define \(f_d: \mathbb{S}_d^{++}\to\mathbb{S}_d^{++},\ X\mapsto X\cdot X\). Since every \(X\in\mathbb{S}_d^{++}\) has full rank, \(f_d\) is well-defined.
 We will show that \(f_d\) is bijective and that, for every \(X\in\mathbb{S}_d^{++}\), the differential \(df_d(X)\) is a linear homeomorphism.

 By \cite[Theorem 7.2.6]{HORMAT}, a positive definite matrix has a unique positive definite square root. Hence \(f_d\) is bijective. Let \(X\in\mathbb{S}_d^{++}\). 
 The differential of \(f_d\) at the point \(X\) is given by
 \[\forall H\in\mathbb{S}_d\qquad df_d(X)H=\left.\frac{\partial}{\partial t}f_d(X+tH)\right|_{t=0}=HX+XH.\]
 Our aim is to show that \(H=0\) whenever \(df_d(X)H=0\). In fact, this is a simple consequence of \cite[Theorem 1]{POTMAT} where the solutions \(X\) of the general
 matrix quadratic equation \(0=A+BX+XB^t-XCX\) for fixed \(A,B,C\in M_d(\mathbb{R})\) have been analysed.
\end{proof}
Since \(X_n = \Sigma_n^{1/2}\epsilon_n\), we thus obtain that 
\begin{align}
 f: U\times\mathbb{R}^d\to\mathbb{R}^d\nonumber,\quad (Y_{n-1},\epsilon_n)\mapsto G(\Sigma_n)\epsilon_n=X_n \nonumber
\end{align}
is a \(C^1\) - map from \(U\times\mathbb{R}^d\) into \(\mathbb{R}^d\) where \(U\) is the open set in 
\(\left(\mathbb{R}^{d(d+1)/2}\right)^p\times\left(\mathbb{R}^d\right)^{q}\) defined by
\[U := \underbrace{\mathrm{vech}(\mathbb{S}_d^{++})\times\ldots\times\mathrm{vech}(\mathbb{S}_d^{++})}_{p}\times
 \underbrace{\mathbb{R}^d\times\ldots\times\mathbb{R}^d}_{q}.\]
Due to the assumption that \(C\) and the initial values $\Sigma_0,\ldots,\Sigma_{1-p}$ are positive definite, every \(\Sigma_n\) and \(\Sigma_n^{1/2}\) is also 
positive definite  and thus \(G(\Sigma_n)=\Sigma_n^{1/2}\) is always an invertible matrix. Hence, for every \(Y\in U\), the map \(f_Y( \cdot )=f(Y, \cdot )\) is 
linear bijective from \(\mathbb{R}^d\) onto \(\mathbb{R}^d\), i.e.\ \(f_Y( \cdot )\) is a \(C^1\) - diffeomorphism. Moreover the map
\(U\times\mathbb{R}^d\to\mathbb{R}^d,\ (Y,\epsilon)\mapsto f_Y^{-1}(\epsilon)\) is continuous in \((Y,\epsilon)\) where \(f_Y^{-1}( \cdot )\) denotes the inverse 
of \(f_Y( \cdot )\).

Altogether we are thus in our setting of semi-polynomial Markov chains. For the  Markovian  representation $Y_n$ of a standard GARCH($p,q$) process $X_n$ we have
\begin{equation}
 Y_n=F(Y_{n-1},\epsilon_n):=\varphi(Y_{n-1},\underbrace{f_{Y_{n-1}}(\epsilon_n)}_{=X_n})
 \label{eq:GARCHstatespace}
\end{equation}
where \(F\) is  a \(C^1\) - map from \(U\times\mathbb{R}^d\) into \(U\). Moreover, it is obvious that \eqref{3.7}, \eqref{3.8} have a stationary solution if and only if                                                                                                                                                                     
\eqref{eq:GARCHstatespace} has one.
\subsection{Some Results from Linear Algebra}
\label{Section 3.3.2}
In this section we will show some results from linear algebra which will be necessary to establish the Foster-Lyapunov-condition for multivariate GARCH models.

Let \(n\in\mathbb{N}^*\) and \((F_i)_{1\leq i\leq n}\) be  elements of \(M_d(\mathbb{R})\). We set
\begin{align}
 \xi: M_d(\mathbb{R})&\to M_d(\mathbb{R}),\,\,\xi(M):=\sum\limits_{i=1}^nF_iMF_i^t. \nonumber
\end{align}
This map is obviously linear. We can consider \(\xi\) a linear map from \(\mathbb{R}^{d^2}\) into \(\mathbb{R}^{d^2}\) using the \(\mathrm{vec}\)
operator as follows :
\[\mathrm{vec}(\xi(M)) 
 = \bigg(\sum\limits_{i=1}^nF_i\otimes F_i\bigg)\mathrm{vec}(M) =F\mathrm{vec}(M) \mbox{ where } F:=\sum\limits_{i=1}^nF_i\otimes F_i.\]

Note that we have \(\xi(\mathbb{S}_d)\subseteq\mathbb{S}_d\), i.e.\ the symmetric \(d\times d\) matrices are mapped into themselves by \(\xi\). We denote 
by \(\tilde{\xi}\) the restriction of \(\xi\) to the linear subspace \(\mathbb{S}_d\). Using the $\mathrm{vech}$ operator, we obtain, for all \(M\in\mathbb{S}_d\),
\begin{align}
 \mathrm{vech}(\tilde{\xi}(M)) &= \mathrm{vech}(\xi(M))\nonumber= H_d\ \mathrm{vec}(\xi(M))= H_dF\ \mathrm{vec}(M)
  = H_dFK_d^t\ \mathrm{vech}(M). \nonumber  
\end{align}
Since we can identify \(\mathbb{S}_d\) via the \(\mathrm{vech}\) operator with \(\mathbb{R}^{d(d+1)/2}\), the transformation matrix of \(\tilde{\xi}\) is given by
\[\tilde{F}:=H_dFK_d^t.\]
We obtain the following lemma:
\begin{lemma}
 Let \(C\in\mathbb{S}_d^{++}\). The following statements are equivalent:
 \begin{enumerate}
  \item
   The spectral radius of \(\xi\) is less than \(1\).
  \item
   The spectral radius of \(\tilde{\xi}\) is less than \(1\).
  \item
   There is \(\Sigma\in\mathbb{S}_d^{++}\) such that \(\Sigma = C + \xi(\Sigma)\).
 \end{enumerate}
 \label{Lemma 3.3.1}
\end{lemma}
\begin{proof}
 (i) \(\Rightarrow\) (ii): Obvious since \(\tilde{\xi}\) is a restriction of \(\xi\). \\
 (ii) \(\Rightarrow\) (iii): If the spectral radius of \(\tilde{\xi}\) is less than \(1\), then the Neumann series \(\sum_{n=0}^{\infty}\tilde{\xi}^n\) is convergent with 
 respect to a suitable operator norm. We define
 \begin{equation}
  \Sigma:=\sum\limits_{n=0}^{\infty}\tilde{\xi}^n(C).
 \label{3.9}
 \end{equation}
 Clearly, \(\Sigma\) is symmetric. Moreover, for all \(M\in\mathbb{S}_d^+\), we have \(\tilde{\xi}(M)\in\mathbb{S}_d^+\) by the definition of 
 \(\tilde{\xi}\) and \(\xi\), respectively. By iteration we obtain that \(\tilde{\xi}^n(M)\in\mathbb{S}_d^+\) for all \(n\in\mathbb{N}^*\). Thus the matrix
 \(\Sigma-\tilde{\xi}^0(C)=\Sigma-C\) is symmetric and positive semi-definite. This implies that \(\Sigma\) is positive definite. \\
 Since \(\xi\) and \(\tilde{\xi}\) coincide on \(\mathbb{S}_d\), we deduce 
 \(\Sigma = \sum\limits_{n=0}^{\infty}\xi^n(C)= C + \xi\Big(\sum\limits_{n=1}^{\infty}\xi^{n-1}(C)\Big)= C + \xi(\Sigma).\) \\
 (iii) \(\Rightarrow\) (i): Suppose that there exists \(\Sigma\in\mathbb{S}_d^{++}\) such that \(\Sigma = C + \xi(\Sigma)\). We denote the complex
 \(d\times d\) matrices by \(M_d(\mathbb{C})\) and the conjugate transpose of a vector \(x\in\mathbb{C}^d\) by \(x^*\). For every \(P\in M_d(\mathbb{C})\) we 
 define
 \[\left\|P\right\|_{\Sigma}:= \sup\limits_{x\in\mathbb{C}^d,\ x^*\Sigma\;x=1}\left|x^*Px\right|\]
 which is a norm on \(M_d(\mathbb{C})\) since \(\Sigma\in\mathbb{S}_d^{++}\).

 Then, for all \(x\in\mathbb{C}^d\), \(\left|x^*Px\right|\leq\left\|P\right\|_{\Sigma}(x^*\Sigma x).\) Since the unit sphere 
 \(\left\{x\in\mathbb{C}^d:\ x^*\Sigma x = 1\right\}\) is compact, there exists, for every \(P\in M_d(\mathbb{C})\), a vector \(x_p\in\mathbb{C}^d\) such that
 \(\left\|P\right\|_{\Sigma}=\left|x_P^*Px_P\right|\text{ with } x_P^*\Sigma x_P = 1.\) Let now \(\lambda\) be an eigenvalue of \(\xi\). Then there is an 
 \(M\in M_d(\mathbb{C}),\ M\neq 0\) such that \(\lambda M = \xi(M) = \sum\limits_{i=1}^nF_iMF_i^t.\) For every \(x\in\mathbb{C}^d\), we deduce
 \begin{align}
  \left|\lambda\right|\cdot\left|x^*Mx\right| &= \bigg|\sum\limits_{i=1}^nx^*F_iMF_i^tx\bigg|\leq
   \sum\limits_{i=1}^n\left|\left(F_i^tx\right)^*M\left(F_i^tx\right)\right|\leq\left\|M\right\|_{\Sigma}\sum\limits_{i=1}^nx^*F_i\Sigma F_i^tx
   =\left\|M\right\|_{\Sigma}x^*\underbrace{\bigg(\sum\limits_{i=1}^nF_i\Sigma F_i^t\bigg)}_{=\xi(\Sigma)=\Sigma-C}x. \nonumber
 \end{align}
 If we choose \(x_M\) such that \(\left\|M\right\|_{\Sigma}=\left|x_M^*Mx_M\right|\) and \(x_M^*\Sigma x_M=1\), we obtain 
 \(\left|\lambda\right|\leq 1-x_M^*Cx_M < 1\) (note that \(\left\|M\right\|_{\Sigma}\neq 0\)). Hence the spectral radius of \(\xi\) is less than 1.
\end{proof}
We consider now the families of matrices \((\bar{A}_{i,k},\bar{B}_{j,r})\) and \((A_i,B_j)\) which occur in the BEKK and \(\mathrm{vech}\) representation 
of a standard GARCH\((p,q)\) model, respectively.
\begin{proposition}
 The spectral radius of \(\sum_{i=1}^qA_i+\sum_{j=1}^pB_j\) is less than \(1\) if and only if there exists \(\Sigma\in\mathbb{S}_d^{++}\) such that
 \begin{equation}
  \Sigma = C + \sum\limits_{i=1}^q\sum\limits_{k=1}^{l_i}\bar{A}_{i,k}\Sigma\bar{A}_{i,k}^t + 
   \sum\limits_{j=1}^p\sum\limits_{r=1}^{s_j}\bar{B}_{j,r}\Sigma\bar{B}_{j,r}^t.
 \label{eq:BEKKfixedpoint}
 \end{equation}
\label{Proposition 3.3.2}
\end{proposition}
\begin{proof}
 Define \(\xi: M_d(\mathbb{R})\to M_d(\mathbb{R})\) by 
 \(\xi(M)=\sum\limits_{i=1}^q\sum\limits_{k=1}^{l_i}\bar{A}_{i,k}M\bar{A}_{i,k}^t+\sum\limits_{j=1}^p\sum\limits_{r=1}^{s_j}\bar{B}_{j,r}M\bar{B}_{j,r}^t.\)
 Then the transformation matrix of \(\xi\) is 
 \[F = \sum\limits_{i=1}^q\sum\limits_{k=1}^{l_i}\bar{A}_{i,k}\otimes\bar{A}_{i,k} + 
  \sum\limits_{j=1}^p\sum\limits_{r=1}^{s_j}\bar{B}_{j,r}\otimes\bar{B}_{j,r} = \sum\limits_{i=1}^q\tilde{A}_i + \sum\limits_{j=1}^p\tilde{B}_j.\]
 Note that the transformation matrix of \(\tilde{\xi}\) (restriction of \(\xi\) to the linear subspace \(\mathbb{S}_d\)) is 
 \(H_dFK_d^t=\sum_{i=1}^qA_i + \sum_{j=1}^pB_j\). Due to Lemma \ref{Lemma 3.3.1} the spectral radius of \(\sum_{i=1}^qA_i+\sum_{j=1}^pB_j\) is less 
 than \(1\) if and only if there is \(\Sigma\in\mathbb{S}_d^{++}\) such that \eqref{eq:BEKKfixedpoint} holds.
\end{proof}
\begin{bemerkung}
 By a simple transposition argument one can equivalently state that the spectral radius of \(\sum_{i=1}^qA_i+\sum_{j=1}^pB_j\) is less than \(1\) if and only if 
 there exists \(\Sigma\in\mathbb{S}_d^{++}\) such that
 \[\Sigma = C + \sum\limits_{i=1}^q\sum\limits_{k=1}^{l_i}\bar{A}_{i,k}^t\Sigma\bar{A}_{i,k} + 
  \sum\limits_{j=1}^p\sum\limits_{r=1}^{s_j}\bar{B}_{j,r}^t\Sigma\bar{B}_{j,r}\]
 which we are going to use in the upcoming proof of Theorem \ref{Satz 3.3.8}.
\label{Bemerkung 3.3.3}
\end{bemerkung}
In the following we consider the block matrix \(B\) defined as in Section \ref{sec:mgarch}.
\begin{proposition}
 \begin{enumerate}
  \item
   If the spectral radius of the matrix \(\sum_{j=1}^pB_j\) is less than \(1\), then the one of \(B\) is also less than \(1\).
  \item
   If the spectral radius of the matrix \(\sum_{i=1}^qA_i+\sum_{j=1}^pB_j\) is less than \(1\), then the one of \(\sum_{j=1}^pB_j\) is also less than \(1\).
 \end{enumerate}
\label{Proposition 3.3.4}
\end{proposition}
\begin{proof}
 (i) Suppose that the spectral radius of \(\sum_{j=1}^pB_j\) is less than \(1\). Then there exists due to Lemma \ref{Lemma 3.3.1} a symmetric positive definite matrix
 \(\tilde{\Sigma}\in\mathbb{S}_d^{++}\) such that 
 \begin{equation}
  \tilde{\Sigma} = C + \sum\limits_{j=1}^p\sum\limits_{r=1}^{s_j}\bar{B}_{j,r}\tilde{\Sigma}\bar{B}_{j,r}^t.
 \label{3.10}
 \end{equation}
 Let \(\lambda\) be an eigenvalue of \(B\) associated with the eigenvector \(h=(h_1^t,\ldots,h_p^t)^t\in\big(\mathbb{R}^{d(d+1)/2}\big)^p.\) Then
 \(\lambda h_1 = \sum_{j=1}^pB_jh_j\text{ and }\lambda h_j = h_{j-1}\ \text{ for }2\leq j\leq p.\) Thus \(h_p\neq 0\) (otherwise \(h\) would be zero) and 
 \(\lambda^ph_p = \lambda(\lambda^{p-1}h_p) = \lambda h_1 = \sum_{j=1}^pB_jh_j = \sum_{j=1}^p\lambda^{p-j}B_jh_p.\) Let \(M\in\mathbb{S}_d\) such 
 that \(\mathrm{vech}(M)=h_p\). Then \(\lambda^pM = \sum_{j=1}^p\sum_{r=1}^{s_j}\lambda^{p-j}\bar{B}_{j,r}M\bar{B}_{j,r}^t.\) We define the norm 
 \(\left\|\cdot\right\|_{\tilde{\Sigma}}\) on \(M_d(\mathbb{C})\) as in the proof of Lemma \ref{Lemma 3.3.1} by 
 \(\left\|P\right\|_{\tilde{\Sigma}}:= \sup\limits_{x\in\mathbb{C}^d,\ x^*\tilde{\Sigma}\;x=1}\left|x^*Px\right|,\ P\in M_d(\mathbb{C}).\) Then, for all
 \(x\in\mathbb{C}^d\), 
 \begin{align}
  \left|\lambda\right|^p\cdot\left|x^*Mx\right| 
   &= \bigg|\sum\limits_{j=1}^p\sum\limits_{r=1}^{s_j}\lambda^{p-j}x^*\bar{B}_{j,r}M\bar{B}_{j,r}^tx\bigg|
    \leq\sum\limits_{j=1}^p\sum\limits_{r=1}^{s_j}\left|\lambda\right|^{p-j}\left|x^*\bar{B}_{j,r}M\bar{B}_{j,r}^tx\right| \nonumber \\
   &\leq\left\|M\right\|_{\tilde{\Sigma}}
    \sum\limits_{j=1}^p\sum\limits_{r=1}^{s_j}\left|\lambda\right|^{p-j}(x^*\bar{B}_{j,r}\tilde{\Sigma}\bar{B}_{j,r}^tx). \nonumber
 \end{align}
 If we assume that there is an eigenvalue \(\lambda\) of \(B\) with \(\left|\lambda\right|\geq 1\), then we obtain, taking the vector \(x\) such
 that \(x^*\tilde{\Sigma}x=1\) and \(\left|x^*Mx\right| = \left\|M\right\|_{\tilde{\Sigma}}\) and using (\ref{3.10}), that
 \begin{align}
  \left|\lambda\right|^p 
   &\leq\sum\limits_{j=1}^p\sum\limits_{r=1}^{s_j}\left|\lambda\right|^{p-j}(x^*\bar{B}_{j,r}\tilde{\Sigma}\bar{B}_{j,r}^tx)
    \leq\left|\lambda\right|^{p-1}\left[x^*\left(\sum\limits_{j=1}^p\sum\limits_{r=1}^{s_j}\bar{B}_{j,r}\tilde{\Sigma}\bar{B}_{j,r}^t\right)x\right] \nonumber \\
   &=\left|\lambda\right|^{p-1}\left[x^*\left(\tilde{\Sigma}-C\right)x\right] = \left|\lambda\right|^{p-1}(1-x^*Cx). \nonumber
 \end{align}
 Since \(C\) is symmetric positive definite, one has \(x^*Cx > 0.\) Hence, \(\left|\lambda\right|^p < \left|\lambda\right|^{p-1}\), i.e.\ \(\left|\lambda\right| < 1\) which 
 is a contradiction. Thus the spectral radius of \(B\) has to be less than \(1\).

 (ii) Suppose that the spectral radius of the matrix \(\sum_{i=1}^qA_i+\sum_{j=1}^pB_j\) is less than \(1\). Then, due to Proposition \ref{Proposition 3.3.2}, there
 exists \(\Sigma\in\mathbb{S}_d^{++}\) such that 
 \(\Sigma = C + \sum_{i=1}^q\sum_{k=1}^{l_i}\bar{A}_{i,k}\Sigma\bar{A}_{i,k}^t + \sum_{j=1}^p\sum_{r=1}^{s_j}\bar{B}_{j,r}\Sigma\bar{B}_{j,r}^t.\)
 We set \(\tilde{C} := C + \sum_{i=1}^q\sum_{k=1}^{l_i}\bar{A}_{i,k}\Sigma\bar{A}_{i,k}^t\). Now, \(\tilde{C}\) is symmetric positive definite and 
 \(\Sigma = \tilde{C} + \sum_{j=1}^p\sum_{r=1}^{s_j}\bar{B}_{j,r}\Sigma\bar{B}_{j,r}^t.\) Using again Proposition \ref{Proposition 3.3.2} we deduce that the
 spectral radius of \(\sum_{j=1}^pB_j\) is less than \(1\).
\end{proof}
\begin{bemerkung}
 The matrix \(\tilde{\Sigma}\) in (\ref{3.10}) is the limit of a Neumann series (cf.\ proof of Lemma \ref{Lemma 3.3.1}).Thus 
 \[\mathrm{vech}(\tilde{\Sigma}) = \bigg(I-\sum_{j=1}^pB_j\bigg)^{-1}\mathrm{vech}(C).\]
\label{Bemerkung 3.3.5}
\end{bemerkung}
\subsection{Verification of Assumption (A2)}
\label{Section 3.3.3}
We will suppose throughout that (H1) holds.
\begin{proposition}
 Suppose that (H2) holds. If the spectral radius of the matrix \(\sum\limits_{j=1}^pB_j\) is less than \(1\), then (A2) holds.
\label{Proposition 3.3.6}
\end{proposition}
\begin{proof}
 Let \(U\) be the open set in \(\left(\mathbb{R}^{d(d+1)/2}\right)^p\times\left(\mathbb{R}^d\right)^{q}\) defined as in Section \ref{sec:mgarchsemipol}.
 For arbitrary \(y\in U\) we define the sequence \((Y_n^y)_{n\in\mathbb{N}}\) by \(Y_0^y=y\text{ and } Y_n^y = F(Y_{n-1}^y,0),\ n\geq 1\).

 We denote by \(X_n^y\) and \(\mathrm{vech}(\Sigma_n^y)\) the associated values of \(X_n\) and \(\mathrm{vech}(\Sigma_n)\). Since, by definition,
 \(X_n = G(\Sigma_n)\epsilon_n = \Sigma_n^{1/2}\epsilon_n\), we obtain that \(X_n^y = 0\) for all \(n\geq 1\). Due to (\ref{3.8}), \(\Sigma_n^y\)
 can be written, for every \(n > q\), as 
 \[\mathrm{vech}(\Sigma_n^y) = \mathrm{vech}(C) + \sum\limits_{j=1}^p B_j\ \mathrm{vech}(\Sigma_{n-j}^y).\]
 Thus, for all \(n > q\) and for all \(y\in U\),
 \begin{equation}
  Y_n^y = \mathscr{C} + \tilde{B} Y_{n-1}^y
 \label{3.11}
 \end{equation}
 where \(\tilde{B}\) is defined as in Section \ref{sec:mgarch}. Due to (\ref{3.11}), the assumption (A2) is satisfied with \(a=0\) if the spectral radius of \(B\) is less than
 \(1\). This is the case, since the spectral radius of \(\sum\limits_{j=1}^pB_j\) is supposed to be less than \(1\) (cf.\ Proposition \ref{Proposition 3.3.4} (i)).                       
\end{proof}
Hence, for all \(y\in U\), the sequence \((Y_n^y)_{n\in\mathbb{N}}\) converges to the unique fixed point \(T\) defined by
\begin{equation}
 T = \mathscr{C} + \tilde{B} T.  
\label{3.12}
\end{equation}
Using Lemma \ref{Lemma 3.3.1} and the fact that the spectral radius of \(\sum_{j=1}^pB_j\) is assumed to be less than \(1\), there is 
\(\tilde{\Sigma}\in\mathbb{S}_d^{++}\) (cf.\ (\ref{3.10})) such that
\[\tilde{\Sigma} = C + \sum\limits_{j=1}^p\sum\limits_{r=1}^{s_j}\bar{B}_{j,r}\tilde{\Sigma}\bar{B}_{j,r}^t.\]
It is then easy to see that \(T\) can be written as 
\begin{equation}
 T = \bigg(\underbrace{\mathrm{vech}(\tilde{\Sigma})^t,\ldots,\mathrm{vech}(\tilde{\Sigma})^t}_{p},\underbrace{0,\ldots,0}_{qd}\bigg)^t\in U.
\label{eq:explformT}
\end{equation}
We set \(\mathscr{C}_1 := \left(\mathrm{vech}(C)^t,0,\ldots,0\right)^t\in\left(\mathbb{R}^{d(d+1)/2}\right)^p\). Then (\ref{3.12}) yields
\begin{equation}
 \sigma = \mathscr{C}_1 + B\sigma
\label{3.13}  
\end{equation}
where \(\sigma := \left(\mathrm{vech}(\tilde{\Sigma})^t,\ldots,\mathrm{vech}(\tilde{\Sigma})^t\right)^t\in\left(\mathbb{R}^{d(d+1)/2}\right)^p\).
\subsection{Verification of Assumption (A3)}
If (H2) is satisfied, then \(E\) contains an open set of \(\mathbb{R}^d\) and we obtain for the algebraic variety of states with the same arguments as in Section
\ref{Section 2.3.2} that
\begin{align}
 W &= \raisebox{0.15cm}{\scriptsize{$Z$}}\overline{S_T} = \raisebox{0.28cm}{\scriptsize{$Z$}}\overline
  {\bigcup\limits_{n\in\left.\mathbb{N}\right.^*}\left.F\right.^n\left(T,\left.E\right.^n\right)}\nonumber=\raisebox{0.28cm}{\scriptsize{$Z$}}\overline
  {\bigcup\limits_{n\in\left.\mathbb{N}\right.^*}\left.F\right.^n\left(T,\left(\mathbb{R}^d\right)^n\right)}\nonumber=\raisebox{0.28cm}{\scriptsize{$Z$}}\overline
  {\bigcup\limits_{n\in\left.\mathbb{N}\right.^*}\varphi^n\left(T,\left(\mathbb{R}^d\right)^n\right)} \nonumber.
\end{align}
Let \(n\in\mathbb{N}^*\) and consider \(y(n)\in\varphi^n\left(T,(\mathbb{R}^d)^n\right)\) given by \(y(n) = \varphi^n(T,x_1,\ldots,x_n)\) where 
\(x_1,\ldots,x_n\in\mathbb{R}^d\). We define \(x(n)\) and \(\sigma(n)\) by the coordinates of \(y(n)\) as follows:
\begin{align}
 x(n) &= \left(\mathrm{vech}(x_nx_n^t)^t,\ldots,\mathrm{vech}(x_{n-q+1}x_{n-q+1}^t)^t\right)^t \nonumber \\
  \text{and}\qquad\sigma(n) &= \left(\mathrm{vech}(\sigma_n)^t,\ldots,\mathrm{vech}(\sigma_{n-p+1})^t\right)^t. \nonumber
\end{align}
That is, \(y(n) = \big(\sigma(n)^t,x_n^t,\ldots,x_{n-q+1}^t\big)^t\). Then
\begin{equation}
 \sigma(n+1) = \mathscr{C}_1 + Ax(n) + B\sigma(n)
\label{3.14}
\end{equation}
where \(\mathscr{C}_1\) and \(B\) are defined in Section \ref{Section 3.3.3} and \(A\) is given by
\[A := \begin{pmatrix} A_1 & A_2 & \ldots & A_q \\ 0   & 0   & \ldots & 0  \\ \vdots & \vdots & \ddots & \vdots \\ 0   & 0   & \ldots & 0  \\ \end{pmatrix}
 \in M_{p\frac{d(d+1)}{2}\times q\frac{d(d+1)}{2}}(\mathbb{R}).\]
Iterating (\ref{3.14}) and due to \(\sigma(0)=\sigma=\left(\mathrm{vech}(\tilde{\Sigma})^t,\ldots,\mathrm{vech}(\tilde{\Sigma})^t\right)^t\) 
(since \(y(0) = T\)) we deduce
\begin{align}
 \sigma(n) &= \underbrace{\sum\limits_{i=0}^{n-1}B^i\mathscr{C}_1 + B^n\sigma}_{\stackrel{\eqref{3.13}}{=}\sigma} + 
  \sum\limits_{i=1}^{n-1}B^{i-1}Ax(n-i)\nonumber= \sigma + \sum\limits_{i=1}^{n-1}B^{i-1}Ax(n-i). \nonumber
\end{align}
This yields \(\mathrm{vech}(\sigma_n) = \mathrm{vech}(\tilde{\Sigma}) + \sum_{i=1}^{n-1}K_i\,\mathrm{vech}(x_{n-i}x_{n-i}^t)\) where, for all
\(i\in\mathbb{N}^*\), \(K_i\) is defined by \(K_i := \left[B^{i-1}A\right]_{1,1} + \left[B^{i-2}A\right]_{1,2} + \ldots + \left[B^{i-q}A\right]_{1,q}\) with the
convention \(B^0:=I\), \(B^i:=0\) if \(i<0\) and \(\left[M\right]_{1,j}\) is the \(d(d+1)/2\times d(d+1)/2\) block from lines \(1\) to \(d(d+1)/2\) and from 
columns \((j-1)d(d+1)/2+1\) to \(jd(d+1)/2\) of \(M\).

Thus, \(W\) is the Zariski closure of the orbit
\begin{align}
 S_T &= \bigcup\limits_{n\in\left.\mathbb{N}\right.^*}\left\{y(n):\ x_1,\ldots,x_n\in\mathbb{R}^d\right\} \nonumber \\
  &=\bigcup\limits_{n\in\left.\mathbb{N}\right.^*}\Bigg\{T + \Bigg(\sum\limits_{i=1}^{n-1}\left(K_i\,\mathrm{vech}(x_{n-i}x_{n-i}^t)\right)^t,
   \ldots,\sum\limits_{i=1}^{n-p}\left(K_i\,\mathrm{vech}(x_{n-p+1-i}x_{n-p+1-i}^t)\right)^t, \nonumber \\
  &\qquad \quad\quad x_n^t,x_{n-1}^t,\ldots,x_{n-q+1}^t\Bigg)^t:\ x_1,\ldots,x_n\in\mathbb{R}^d\Bigg\} \nonumber
\end{align}
where \(x_{1-q} = x_{2-q} = \ldots = x_0 = 0\) (since \(y(0) = T = (\sigma^t,0,\ldots,0)^t\)). \\
In particular, this implies \(\varphi(W\cap U\times\mathbb{R}^d)\subseteq W\cap U\) (cf. (F2)), because $\varphi(U\times \mathbb{R}^d)\subseteq U$ and 
\(\varphi(S_T\times\mathbb{R}^d)\subseteq S_T\) yields 
\[\varphi(W\cap U\times\mathbb{R}^d) \subseteq \varphi\left(\raisebox{0.15cm}{\scriptsize{$Z$}}\overline{S_T}\times\mathbb{R}^d\right)
  \subseteq\raisebox{0.23cm}{\scriptsize{$Z$}}\overline{\varphi\left(\raisebox{0.15cm}{\scriptsize{$Z$}}\overline{S_T}\times\mathbb{R}^d\right)}
  =\raisebox{0.18cm}{\scriptsize{$Z$}}\overline{\varphi(S_T\times\mathbb{R}^d)}\subseteq\raisebox{0.15cm}{\scriptsize{$Z$}}\overline{S_T}=W,\]
since \(\varphi\) is a regular map and thus continuous with respect to the Zariski topology.
\begin{satz}
 Suppose that (H2) holds and that there is a strictly stationary solution \((X_n)_{n\in\mathbb{Z}}\) for the standard  GARCH\((p,q)\) model, then the process
 \((Y_n)_{n\in\mathbb{Z}}\) takes its values in the algebraic variety of states \(W\cap U\). Moreover, one has
 \[\mathrm{vech}(\Sigma_n) = \bigg(I-\sum\limits_{j=1}^pB_j\bigg)^{-1}\mathrm{vech}(C) + \sum\limits_{i=1}^{\infty}K_i\,\mathrm{vech}(X_{n-i}X_{n-i}^t).\]
\label{Satz 3.3.7}
\end{satz}
\begin{proof}
 Let \((X_n)_{n\in\mathbb{Z}}\) be a strictly stationary solution of the standard GARCH\((p,q)\) model with conditional covariance matrices \(\Sigma_n\). We denote 
 by \(X(n)\) and \(\Sigma(n)\) the following random vectors:
 \begin{align}
  X(n) &= \left(\mathrm{vech}(X_nX_n^t)^t,\ldots,\mathrm{vech}(X_{n-q+1}X_{n-q+1}^t)^t\right)^t \nonumber \\
   \text{and}\qquad\Sigma(n) &= \left(\mathrm{vech}(\Sigma_n)^t,\ldots,\mathrm{vech}(\Sigma_{n-p+1})^t\right)^t. \nonumber
 \end{align}
 Since \(\Sigma(n) = \mathscr{C}_1 + AX(n-1) + B\Sigma(n-1)\) (cf.\ (\ref{3.14})), iterating yields 
 \begin{equation}
  \Sigma(n) = \sum\limits_{i=0}^{k-1}B^i\mathscr{C}_1 + B^k\Sigma(n-k) + \sum\limits_{i=1}^{k}B^{i-1}AX(n-i)      
 \label{3.15}
 \end{equation}
 for all \(k\in\mathbb{N}\).

 Now for any  \(M=\left(\mathrm{vech}(M_1)^t,\ldots,\mathrm{vech}(M_p)^t\right)^t\) and \(N=\left(\mathrm{vech}(N_1)^t,\ldots,\mathrm{vech}(N_p)^t\right)^t\)
 in \((\mathbb{R}^{d(d+1)/2})^p\), let us denote \(M\geq N\) if and only if \(M_1\geq N_1,\ldots,M_p\geq N_p\) (where, for all 
 \(M_i,N_i\in\mathbb{S}_d\), \(M_i\geq N_i\Leftrightarrow M_i-N_i\geq 0\Leftrightarrow M_i-N_i\) positive semi-definite). This defines a partial order on
 \((\mathbb{R}^{d(d+1)/2})^p\).

 Then (\ref{3.15}) yields \(\Sigma(n)\geq\sum_{i=0}^{k-1}B^i\mathscr{C}_1\). Since \(\Sigma(n)\) is finite the series \(\sum_{i=0}^{k-1}B^i\mathscr{C}_1\)
 converges as \(k\to\infty\) (see for instance \cite{WARPAR} for further details concerning partially ordered topological spaces; in particular the Corollary after Lemma 5
 proves that our series must converge). Setting \(\tilde{\sigma}:=\sum_{i=0}^\infty B^i\mathscr{C}_1\), it is easy to see that  
 \(\tilde{\sigma}=\mathscr{C}_1+B\tilde{\sigma}\). Using the definitions of \(B\) and \(\mathscr{C}_1\), we obtain that 
 \(\tilde{\sigma}=\left(\sigma_1^t,\sigma_1^t,\ldots,\sigma_1^t\right)^t\) for some \(\sigma_1\in\mathbb{R}^{d(d+1)/2}\) which fulfils 
 \(\sigma_1 = \mathrm{vech}(C) + \sum_{j=1}^pB_j\sigma_1\). One may then verify that \(\sigma_1 = \mathrm{vech}(\Sigma_1)\) for some
 \(\Sigma_1\in\mathbb{S}_d^{++}\) and hence that the spectral radius of \(\sum_{j=1}^pB_j\) is less than \(1\) (cf.\ Proposition \ref{Proposition 3.3.2}). Due to
 Proposition \ref{Proposition 3.3.4} (i) we obtain that the spectral radius of \(B\) is also less than \(1\). Thus \(\tilde{\sigma}=\sigma\).

 Next, since the spectral radius of \(B\) is less than \(1\), the sequence \((B^k)_{k\in\mathbb{N}}\) converges to zero as \(k\to\infty\). The random vectors 
 \((\Sigma(n-k))_{k\in\mathbb{N}}\) have a constant law because \((X_n)_{n\in\mathbb{Z}}\) is supposed to be a strictly stationary solution of the GARCH model. 
 Thus \(B^k\Sigma(n-k)\) converges to zero in probability when \(k\to\infty\).

 With an analog argument as for \(\sum_{i=0}^{k-1}B^i\mathscr{C}_1\) one can see that \(\sum_{i=1}^{k}B^{i-1}AX(n-i)\) converges almost surely as \(k\to\infty\).
 Hence, taking the limit of (\ref{3.15}) yields 
 \[\Sigma(n) = \sigma + \sum\limits_{i=1}^{\infty}B^{i-1}AX(n-i)\qquad\text{a.s.}\]
 Using the matrices \(K_i\), defined during the investigation of the variety of states \(W\), we obtain
 \[\mathrm{vech}(\Sigma_n) = \mathrm{vech}(\tilde{\Sigma}) + \sum\limits_{i=1}^{\infty}K_i\ \mathrm{vech}(X_{n-i}X_{n-i}^t)\qquad\text{a.s.}\]
 This shows that \((Y_n)_{n\in\mathbb{Z}}\) takes its values in the variety \(W\) and hence in \(W\cap U\). Note that the strictly stationary solution is causal.
 To finish the proof we refer to Remark \ref{Bemerkung 3.3.5} from which we obtain 
 \(\mathrm{vech}(\tilde{\Sigma}) = (I-\sum_{j=1}^pB_j)^{-1}\ \mathrm{vech}(C)\).
\end{proof}
\subsection[Foster - Lyapunov Condition]{Foster - Lyapunov Condition (\ref{Foster-Lyapounov})}
We now derive a function \(V\) satisfying the Foster-Lyapunov-condition  provided that the spectral radius of \(\sum_{i=1}^qA_i+\sum_{j=1}^pB_j\) is less than \(1\).
That is, we prove the following theorem:
\begin{satz}
 Suppose that the spectral radius of the matrix \(\sum_{i=1}^qA_i+\sum_{j=1}^pB_j\) is less than \(1\). Then there exist a function \(V: U\to [1,\infty)\) 
 and positive constants \(\alpha < 1,\ b < \infty\) as well as a Borel set \(K\) in $W\cap U$ such that the (\ref{Foster-Lyapounov}) - condition is satisfied, i.e. there are
 positive constants \(\alpha<1,\ b<\infty\) such that
 \begin{equation*}
  PV(x)\leq\alpha\cdot V(x) + b\cdot{1}_K(x)\, \forall\, x\in W\cap U.
 \end{equation*}
\label{Satz 3.3.8}
\end{satz}
\begin{proof}
 For notational convenience we suppose that in the BEKK representation (\ref{3.6}) \(l_i=s_j=1\) for all \(i=1,\ldots,q\) and \(j=1,\ldots,p\), since the extension 
 to general \(l_i,s_j\not=1\) is obvious and trivial. We set\enlargethispage{0.5cm} \(\bar{A}_i:=\bar{A}_{i,1}\) and \(\bar{B}_j:=\bar{B}_{j,1}\). That is, we have
 \begin{equation}
  \Sigma_n = C + \sum\limits_{i=1}^q\bar{A}_iX_{n-i}X_{n-i}^t\bar{A}_i^t + \sum\limits_{j=1}^p\bar{B}_j\Sigma_{n-j}\bar{B}_j^t.
 \label{3.16}
 \end{equation}
 If the spectral radius of the matrix \(\sum_{i=1}^qA_i+\sum_{j=1}^pB_j\) is less than \(1\), then, due to Proposition \ref{Proposition 3.3.2} and Remark
 \ref{Bemerkung 3.3.3}, there exists \(\Sigma\in\mathbb{S}_d^{++}\) such that 
 \(\Sigma = C + \sum_{i=1}^q\bar{A}_i^t\Sigma\bar{A}_i + \sum_{j=1}^p\bar{B}_j^t\Sigma\bar{B}_j.\)

 We define the map \(V: U\to\left[1,\infty\right)\) by 
 \[V(Y_n) := \mathrm{tr}(V_1\Sigma_n) + \ldots + \mathrm{tr}(V_p\Sigma_{n-p+1}) + X_n^tV_{p+1}X_n + \ldots + X_{n-q+1}^tV_{p+q}X_{n-q+1}+1\]
 where \(\mathrm{tr}(\hspace{.5mm}\cdot\hspace{.5mm})\) denotes the trace of a matrix and the \(d\times d\) matrices \((V_i)_{1\leq i\leq p+q}\) are given by
 \begin{align}
  V_k &:= \frac{p-k+1}{p+q}C + \sum\limits_{j=k}^p\bar{B}_j^t\Sigma\bar{B}_j,\quad 1\leq k\leq p \nonumber \\
   V_{p+k} &:= \frac{q-k+1}{p+q}C + \sum\limits_{i=k}^q\bar{A}_i^t\Sigma\bar{A}_i,\quad 1\leq k\leq q. \nonumber
 \end{align}
 Setting \(y=\left(\mathrm{vech}(\Sigma_{n-1})^t,\ldots,\mathrm{vech}(\Sigma_{n-p})^t,X_{n-1}^t,\ldots,X_{n-q}^t\right)^t\in U\) we obtain
 \begin{align}
  \mathbb{E}[V(Y_n)|Y_{n-1}=y]=&\mathbb{E}\left[\mathrm{tr}(V_1\Sigma_n) + X_n^tV_{p+1}X_n|Y_{n-1}=y\right] + \mathrm{tr}(V_2\Sigma_{n-1}) +
   \ldots + \mathrm{tr}(V_p\Sigma_{n-p+1}) \nonumber \\
   &+ X_{n-1}^tV_{p+2}X_{n-1} + \ldots + X_{n-q+1}^tV_{p+q}X_{n-q+1} + 1. \label{3.17}
 \end{align}
 Using (\ref{3.16}) for \(\Sigma_n\), we deduce for the first term at the right hand side
 \begin{align}
  \mathbb{E}&\left[\mathrm{tr}(V_1\Sigma_n) + X_n^tV_{p+1}X_n|Y_{n-1}=y\right] \nonumber \\
   &= \mathbb{E}\left[X_n^tV_{p+1}X_n|Y_{n-1}=y\right] + \mathrm{tr}(V_1C)+\mathrm{tr}(V_1\bar{A}_1X_{n-1}X_{n-1}^t\bar{A}_1^t) + \ldots 
    + \mathrm{tr}(V_1\bar{A}_qX_{n-q}X_{n-q}^t\bar{A}_q^t) \nonumber \\
   &\qquad+\mathrm{tr}(V_1\bar{B}_1\Sigma_{n-1}\bar{B}_1^t) + \ldots + \mathrm{tr}(V_1\bar{B}_p\Sigma_{n-p}\bar{B}_p^t) \nonumber \\
   &= \mathbb{E}\left[X_n^tV_{p+1}X_n|Y_{n-1}=y\right] + \mathrm{tr}(V_1C) + X_{n-1}^t\bar{A}_1^tV_1\bar{A}_1X_{n-1} + \ldots 
    + X_{n-q}^t\bar{A}_q^tV_1\bar{A}_qX_{n-q} \nonumber \\
   &\qquad+\mathrm{tr}(\bar{B}_1^tV_1\bar{B}_1\Sigma_{n-1}) + \ldots + \mathrm{tr}(\bar{B}_p^tV_1\bar{B}_p\Sigma_{n-p}). \nonumber
 \end{align}
 Since \(X_n=\Sigma_n^{1/2}\epsilon_n,\ \Sigma_n^{1/2}\Sigma_n^{1/2}=\Sigma_n\) and \(\mathbb{E}[\epsilon_n\epsilon_n^t]={I}_d\), we obtain 
 \begin{align}
  \mathbb{E}&\left[X_n^tV_{p+1}X_n|Y_{n-1}=y\right] = \mathbb{E}\left[\mathrm{tr}(X_n(V_{p+1}X_n)^t)|Y_{n-1}=y\right] \nonumber \\
   &=\mathrm{tr}\left(\mathbb{E}\left[X_nX_n^tV_{p+1}|Y_{n-1}=y\right]\right) = 
    \mathrm{tr}\left(\mathbb{E}\left[X_nX_n^t|Y_{n-1}=y\right]V_{p+1}\right) = \mathrm{tr}(\Sigma_nV_{p+1}) \nonumber \\
   &= \mathrm{tr}(V_{p+1}C) + \mathrm{tr}(V_{p+1}\bar{A}_1X_{n-1}X_{n-1}^t\bar{A}_1^t) + \ldots 
    + \mathrm{tr}(V_{p+1}\bar{A}_qX_{n-q}X_{n-q}^t\bar{A}_q^t) \nonumber \\
   &\qquad+\mathrm{tr}(V_{p+1}\bar{B}_1\Sigma_{n-1}\bar{B}_1^t) + \ldots + \mathrm{tr}(V_{p+1}\bar{B}_p\Sigma_{n-p}\bar{B}_p^t) \nonumber \\
   &= \mathrm{tr}(V_{p+1}C) + X_{n-1}^t\bar{A}_1^tV_{p+1}\bar{A}_1X_{n-1} + \ldots 
    + X_{n-q}^t\bar{A}_q^tV_{p+1}\bar{A}_qX_{n-q} \nonumber \\
   &\qquad+\mathrm{tr}(\bar{B}_1^tV_{p+1}\bar{B}_1\Sigma_{n-1}) + \ldots + \mathrm{tr}(\bar{B}_p^tV_{p+1}\bar{B}_p\Sigma_{n-p}). \nonumber
 \end{align}
 Hence, (\ref{3.17}) can be rewritten as 
 \begin{align}
  \mathbb{E}&[V(Y_n)|Y_{n-1}=y] \nonumber \\
   &= \mathrm{tr}\left[\left(\bar{B}_1^t(V_1+V_{p+1})\bar{B}_1+V_2\right)\Sigma_{n-1}\right] + \ldots 
    + \mathrm{tr}\left[\left(\bar{B}_{p-1}^t(V_1+V_{p+1})\bar{B}_{p-1}+V_p\right)\Sigma_{n-p+1}\right] \nonumber \\
   &\quad+ \mathrm{tr}\left[\bar{B}_p^t(V_1+V_{p+1})\bar{B}_p\Sigma_{n-p}\right]
    + X_{n-1}^t\left(\bar{A}_1^t(V_1+V_{p+1})\bar{A}_1+V_{p+2}\right)X_{n-1} + \ldots \nonumber \\
   &\quad+ X_{n-q+1}^t\left(\bar{A}_{q-1}^t(V_1+V_{p+1})\bar{A}_{q-1}+V_{p+q}\right)X_{n-q+1} 
    + X_{n-q}^t\bar{A}_q^t(V_1+V_{p+1})\bar{A}_qX_{n-q} \nonumber \\
   &\quad+ \mathrm{tr}[(V_1+V_{p+1})C] + 1. \nonumber
 \end{align}
 By definition of \(V_i\), we deduce
 \begin{align}
  &\bar{B}_k^t(V_1+V_{p+1})\bar{B}_k+V_{k+1} = V_k - \frac{C}{p+q},\quad 1\leq k\leq p-1 \nonumber \\
  &\bar{B}_p^t(V_1+V_{p+1})\bar{B}_p = V_p - \frac{C}{p+q} \nonumber \\
  &\bar{A}_k^t(V_1+V_{p+1})\bar{A}_k+V_{p+k+1} = V_{p+k} - \frac{C}{p+q},\quad 1\leq k\leq q-1 \nonumber \\
  &\bar{A}_q^t(V_1+V_{p+1})\bar{A}_q = V_{p+q} - \frac{C}{p+q}. \nonumber
 \end{align}
 Furthermore, \(V_k\) is symmetric positive definite for all \(k=1,\ldots,p+q\) which implies that \(V_k-\frac{C}{p+q}\) is symmetric positive semi-definite for 
 all \(k=1,\ldots,p+q\).

 Consider the non-negative constants \((\alpha_k)_{1\leq k\leq p+q}\) defined by
 \[\alpha_k := \max\left\{x^t\left(V_k - \frac{C}{p+q}\right)x:\ x\in\mathbb{R}^d,\,x^tV_kx=1\right\}.\]
 Since the maximum is calculated over the unit sphere with respect to the induced norm by \(V_k\) and since this unit sphere is compact, there exists
 \(x_k\in\mathbb{R}^d\) such that \(x_k^tV_kx_k=1\) and
 \begin{align}
  \alpha_k &= x_k^t\left(V_k - \frac{C}{p+q}\right)x_k= 1 - x_k^t\frac{C}{p+q}x_k. \nonumber
 \end{align}
 The matrices \(V_k\), \(k=1,\ldots,p+q\), and \(\frac{C}{p+q}\) are positive definite which yields \(0\leq\alpha_k<1\) for all \(k=1,\ldots,p+q\).

 Setting \(\alpha_0:=\max\left\{\alpha_k:\ k=1,\ldots,p+q\right\}\) we obtain \(0\leq\alpha_0<1\) and 
 \[ V_k - \frac{C}{p+q}\leq\alpha_0 V_k\, \forall \,k\in\left\{1,\ldots,p+q\right\}.\]
 Hence, for all \(M\in\mathbb{S}_d^{++}\) and all \(k\in\left\{1,\ldots,p+q\right\}\),
 \[\mathrm{tr}\left[\left(V_k - \frac{C}{p+q}\right)M\right]\leq\alpha_0\ \mathrm{tr}(V_kM).\]
 We deduce \(\mathbb{E}[V(Y_n)|Y_{n-1}=y]\leq\alpha_0V(y) + \mathrm{tr}(\Sigma C) + 1 - \alpha_0.\)

 If we choose \(\alpha:=(\alpha_0 + 1)/2\in\left[1/2,1\right)\) and \(b:=\mathrm{tr}(\Sigma C) + 1 - \alpha_0\in\left(0,\infty\right)\), then the 
 (\ref{Foster-Lyapounov}) - condition is satisfied with the  set \(K\) given by
 \[K:=\left\{x\in W\cap U:\ V(x)\leq\frac{b}{\alpha-\alpha_0}\right\}.\]
\end{proof}
\subsection[Proof of Main Theorem]{Proof of Theorem \ref{th:maintheo}}
Now we can prove our stationarity and ergodicity result for standard GARCH($p,q$) processes. The main remaining problem is that $K$ is not compact and, hence, it is
somewhat tricky to prove that it is small.
\begin{proof}[{Proof of Theorem \ref{th:maintheo} .}]
 (i) Since due to Proposition \ref{Proposition 3.3.4} (ii) the spectral radius of \(\sum_{j=1}^pB_j\) is also less than \(1\), Proposition \ref{Proposition 3.3.6} and 
 Theorem \ref{Satz 3.3.7} imply that (A2) and (A3) hold. Using then Proposition \ref{Proposition 2.3.4} we deduce that \((Y_n)_{n\in\left.\mathbb{N}\right.^*}\) is
 \(\psi\)-irreducible and aperiodic on the state space \((W\cap U,\mathcal{B}(W\cap U))\).

 Define $U_C= \left(\mathrm{vech}(\{x\in\mathbb{S}_d^{++}:\,x\geq C\})\right)^p\times (\mathbb{R}^d)^q$ which is a closed set and a proper subset of $U$. 
 Then we have by inspecting the iteration that $Y_k\in W\cap U_C \,\text{ for all } \, Y_0\in W\cap U \text{ and } k\in \mathbb{N},\, k\geq p.$ By the way, $T\in U_C$ 
 by \eqref{eq:explformT}. By condition (H3), Theorem \ref{Satz 3.3.8} ensures the existence of a function \(V\) which fulfils the (\ref{Foster-Lyapounov}) - condition 
 on the set \(K\). Now we show that $K$ is small.

 Det $K_1=K\cap U_C^c$ and $K_2=K\backslash K_1$. Using the self-duality of the cone of positive semi-definite matrices, it is straightforward to see that $V$ maps
 unbounded (with respect to norms on $(\bbr^{d(d+1)/2})^p\times (\bbr^d)^q$) subsets of $U$ to unbounded subsets of $\mathbb{R}^+$ and thus $K$ is a bounded
 subset of $U$. Inspecting the iteration defining the GARCH processes further we see that $Y_p$ is not only in $W\cap U_C$ when $Y_0\in K_1$, but necessarily also in a
 compact set $\tilde K\subseteq W\cap U_C$ conditional on $\|\epsilon_i\|\leq \eta$ for $i=1, 2, \ldots, p$ and a fixed $\eta>0$. W.l.o.g. one can assume 
 $\tilde K\supseteq K_2$. This implies $P^p(x,\tilde K)\geq P(\|\epsilon_1\|\leq \eta)^p=:\zeta>0$ for all $x\in K_1$ due to (H2). 

 Moreover, the Markov chain \((Y_n)_{n\in\left.\mathbb{N}\right.^*}\) has the Feller property, as an elementary and standard dominated convergence argument shows,
 and \(\mathrm{supp}\ \psi\) has non-empty interior (see Proposition \ref{Proposition 2.3.4}). Thus, \cite[Proposition 6.2.8]{Meynetal1993} shows that $\tilde K$ is 
 petite (see \cite{MeynTweedie1992}), i.e. there is a non-degenerate measure $\nu$ on $\mathcal{B}(W\cap U)$ and a probability measure $a$ on $\mathbb{N}^*$ 
 such that $\sum_{i=1}^\infty a(\{i\})P^i(x,B)\geq \nu(B)$ for all $x\in\tilde K$ and Borel sets $B\subseteq W\cap U$. Using Chapman-Kolmogorov this implies
 \(0.5\sum_{i=1}^\infty a(\{i\})P^i(x,B)+0.5\sum_{i=p+1}^\infty a(\{i-p\})P^i(x,B)\geq 0.5 \zeta \nu(B)\) for all $x\in \tilde K\cup K_1$ and Borel sets 
 $B\subseteq W\cap U$. Thus $\tilde K\cup K_1$ is petite. Since $K\subseteq \tilde K\cup K_1$, also $K$ is petite and thus small by \cite[Theorem 5.5.7]{Meynetal1993}.

 Applying Theorem \ref{Satz 2.3.5} and Theorem \ref{Satz 2.3.6} we obtain the claimed positive Harris recurrence, geometric ergodicity as well as geometric 
 $\beta$-mixing and $\pi(V)<\infty$ for the stationary distribution $\pi$. 

 Let $(X_n)_{n\in\mathbb{Z}}$ now be the unique stationary GARCH process. Then $\pi(V)<\infty$ implies
 \[\mathbb{E}[X_n^tV_{p+1}X_n]\leq\mathbb{E}[V(Y_n)]=\pi(V)<\infty\,\forall \,n\in\mathbb{Z}\]
 by definition of \(V\) (cf.\ proof of Theorem \ref{Satz 3.3.8}). This shows that \(X_n\in L^2\) for all \(n\in\mathbb{Z}\).

 Since \(\mathbb{E}[\Sigma_n]=\mathbb{E}[X_nX_n^t]\), we deduce \(\mathbb{E}[\Sigma_n]<\infty\). Using the diagonal dominance property of a positive
 semi-definite matrix ($|m_{ij}|\leq 0.5 (m_{ii}+m_{jj})$ for  $M=(m_{ij})_{1\leq i,j\leq d}\in \mathbb{S}_d^+$),  this implies \(\Sigma_n\in L^1\) for all
 \(n\in\mathbb{Z}\).

 (ii) We now assume that there is a weakly stationary solution for the standard  GARCH\((p,q)\) model. Then \(\Sigma:=\mathbb{E}[X_nX_n^t]\) is well-defined.
 Since \(\Sigma=\mathbb{E}[\Sigma_n]\), taking the expectation in (\ref{3.6}) on both sides yields
 \[\Sigma = C + \sum\limits_{i=1}^q\sum\limits_{k=1}^{l_i}\bar{A}_{i,k}\Sigma\bar{A}_{i,k}^t + 
  \sum\limits_{j=1}^p\sum\limits_{r=1}^{s_j}\bar{B}_{j,r}\Sigma\bar{B}_{j,r}^t.\]
 Due to Proposition \ref{Proposition 3.3.2} the spectral radius of the matrix \((\sum_{i=1}^qA_i+\sum_{j=1}^pB_j)\) has to be less than \(1\).
\end{proof}
Note that the proof shows that $\pi$ is concentrated on $W\cap U_C$, so in the stationary regime the GARCH covariance matrices are always bigger than or equal to $C$.
%
%
\appendix
\section{Algebraic Geometry}
In this appendix we summarise  the necessary details of algebraic geometry to understand the statement of our main result. For more details and comprehensive 
treatments we refer to \cite{BENALGSET,MUMALGGEO}.

We denote by \(\mathbb{R}[X_1,\ldots,X_n]\) the polynomial ring in \(n\) variables formed from the set of polynomials in the variables \(X_1,\ldots,X_n\) with 
coefficients in the field \(\mathbb{R}\). 
\begin{definition}
 \begin{enumerate}
  \item  
   A subset \(V\subseteq\mathbb{R}^n\) is called \emph{semi-algebraic} if it admits some representation of the form
   \[V=\bigcup\limits_{i=1}^s\bigcap\limits_{j=1}^{r_i}\left\{x\in\mathbb{R}^n: P_{i,j}(x)\ \sim_{ij}\ 0\right\},\]
   where, for all \(i=1,\ldots,s\) and \(j=1,\ldots,r_i\),
   \begin{enumerate}
    \item
     \(\sim_{ij}\ \in\left\{>,=,<\right\}\)
    \item
     \(P_{i,j}(X)\in\mathbb{R}[X],\ X=(X_1,\ldots,X_n)\).
   \end{enumerate}
  \label{Semi-algebraische Menge}
  \item  
   A subset \(V\subseteq\mathbb{R}^n\) is called \emph{algebraic} if it can be represented as
   \[V=\left\{x\in\mathbb{R}^n: P_1(x)=\ldots=P_k(x)=0\right\}\]
   where \(k\in\mathbb{N}^*\) and \(P_i(X)\in\mathbb{R}[X_1,\ldots,X_n]\) for all \(i=1,\ldots,k\).
  \label{Algebraische Menge}
 \end{enumerate}
\end{definition}
\begin{bemerkung}
 Real algebraic sets can be represented by one single polynomial, namely, if \(V=\left\{P_1=\ldots=P_k=0\right\}\), then we can take \(P:=P_1^2+\ldots+P_k^2\).
\label{Bemerkung 1.2.3}
\end{bemerkung}
\begin{definition}[``Zariski topology'']
 The topology over $\mathbb{R}^n$ for which the algebraic sets in \(\mathbb{R}^n\) are the closed sets is called the \emph{Zariski topology}. 
\label{Satz 1.2.6}
\end{definition}
\begin{bemerkung}
 \begin{enumerate}
  \item
   The Zariski topology is not Hausdorff (i.e. it does not separate points).
  \item
   Every Zariski closed set in \(\mathbb{R}^n\) is also closed in the usual topology on $\mathbb{R}^n$. Thus, the usual topology is finer than the Zariski topology.
  \item
   We define the Zariski closure of a set \(A\) by \(\raisebox{0.15cm}{\scriptsize{$Z$}}\overline{A}:=\bigcap\limits_{\substack{B\text{ Zariski closed}\\B\supseteq A}}B\).
 \end{enumerate}
\label{Bemerkung 1.2.7}
\end{bemerkung}
\begin{definition}
 An algebraic set \(V\subseteq\mathbb{R}^n\) is said to be \emph{irreducible} if it cannot be decomposed as \(V=V_1\cup V_2\), where both \(V_1\) and \(V_2\) are
 algebraic sets and \(V_1\neq V\) and \(V_2\neq V\). \\
 If \(V\) is an irreducible algebraic set, it is also called \emph{algebraic variety}.
\label{Algebraische Varietaet}
\end{definition}
\begin{definition}
 Let \(V\subseteq\mathbb{R}^n\) be an algebraic variety and define the \emph{ideal of} $V$ by 
 \[I(V):=\left\{P\in\bbr[X_1,\ldots,X_n]:\,P(x)=0\quad\forall x\in V\right\}.\] 
 It is an easy consequence of the Hilbert Basis Theorem (cf. for example \cite{LANALG}) that the ideal $I(V)$ has to be finitely generated, i.e. there exist $l\in\bbn^*$ 
 and $Q_1,\ldots,Q_l\in I(V)$ such that $I(V)$ is the ideal generated by these polynomials. We then call 
 \[\rho(V) := \sup\limits_{x\in V} rank\left(\frac{\partial Q_i}{\partial x_j}(x)\right)_{\shortstack{{\scriptsize{$1\leq i\leq l$}}\\{\scriptsize{$1\leq j\leq n$}}}}\]
 the \emph{rank} of the ideal $I(V)$.

 A point \(x_0\in V\) is said to be a \emph{regular point} of V if
 $\rho(V)=rank\left(\frac{\partial Q_i}{\partial x_j}(x_0)\right)_{\shortstack{{\scriptsize{$1\leq i\leq l$}}\\{\scriptsize{$1\leq j\leq n$}}}}$. Otherwise \(x_0\) 
 is called a \emph{singular point} of $V$. We write \(\mathcal{R}(V)\) to denote the set of regular points of \(V\) and \(\mathcal{S}(V)\) for the set of singular points.
\label{Regulaerer Punkt}
\end{definition}
A natural class of maps are those such that preimages of algebraic sets are again algebraic, i.e.\ maps which are continuous with respect to the Zariski topology.
\begin{definition}
 Let \(V\subseteq\mathbb{R}^n\) and \(W\subseteq\mathbb{R}^m\) be algebraic varieties. Then \(f: V\rightarrow W\) is said to be a \emph{regular map}, if all its
 components \((f_i)_{1\leq i\leq m}\) are regular functions, i.e., for all \(i=1,\ldots,m\), there exist \(P_i, Q_i\in\mathbb{R}[X_1,\ldots,X_n]\) such that
 \[V\cap\left\{x\in\mathbb{R}^n: Q_i(x)=0\right\}=\emptyset\quad\text{and}\quad f_i(x)=\frac{P_i(x)}{Q_i(x)}\quad\forall x\in V.\]
\label{Regulaere Abbildung}
\end{definition}
\begin{proposition}
 Let \(V\subseteq\mathbb{R}^n\), \(W\subseteq\mathbb{R}^m\) be algebraic varieties and \(f: V\rightarrow W\) a regular map. Then f is continuous with respect 
 to the Zariski topology.
\label{Proposition 1.2.17}
\end{proposition}
\section{Theory of Markov Chains}
In this appendix we recall the theorems for Markov chains used in the proof of Theorem \ref{Satz 1.1.18}. To this end let $(X_t)_{t\in\bbn}$ be a Markov chain on 
the state space $(S, \scrb(S))$ with transition probability kernel $P$.
\begin{satz}[cf. \cite{Meynetal1993}, Theorem 14.2.2]$~~$\\
 Suppose that the non-negative functions $V,\,f,\,s$ satisfy the relationship
 \[PV(x)\leq V(x)-f(x)+s(x)\,\forall x\in S,\]
 then, for each $x\in S$ and any stopping time $\tau$, we have
 \[\bbE_x\bigg[\sum_{k=0}^{\tau-1} f(X_k)\bigg]\leq V(x) + \bbE_x\bigg[\sum_{k=0}^{\tau-1} s(X_k)\bigg].\]
\label{Theorem B.1}
\end{satz}
\begin{proposition}[cf. \cite{Meynetal1993}, Proposition 9.1.7 (ii)]$~~$\\
 Suppose that $(X_t)_{t\in\bbn}$ is $\psi$-irreducible. If there exists some petite set $C\in\scrb(S)$ such that $L(x,C)=1$ for all $x\in S$, then $(X_t)_{t\in\bbn}$ 
 is Harris recurrent. 
\label{Proposition B.2}
\end{proposition}
\begin{satz}[cf. \cite{Meynetal1993}, Theorem 14.0.1]$~~$\\
 Suppose that the chain $(X_t)_{t\in\bbn}$ is $\psi$-irreducible and aperiodic and let $f\geq 1$ be a function on $S$. Then the following conditions are equivalent:
 \begin{enumerate}[(i)]
  \item
   The chain is positive recurrent with invariant probability measure $\pi$ and
   \[\pi(f)=\int_S\pi(dx)f(x)<\infty.\]
  \item
   There exist some petite set $C\in\scrb(S)$, a positive constant $b<\infty$ and some extended-valued non-negative function $V$ satisfying $V(x_0)<\infty$ for some
   $x_0\in S$ and
   \[\Delta V(x):=PV(x)-V(x)\leq -f(x) + b\mathds{1}_C(x)\,\forall x\in S.\]
 \end{enumerate}
\label{Theorem B.3}
\end{satz}
\begin{satz}[cf. \cite{Meynetal1993}, Theorem 15.0.1]$~~$\\
 Suppose that the chain $(X_t)_{t\in\bbn}$ is $\psi$-irreducible and aperiodic. If there exist a petite set $C\in\scrb(S)$, constants $b<\infty,\,\beta>0$ and an
 extended-valued function $V\geq 1$ finite at some $x_0\in S$ satisfying
 \[\Delta V(x)\leq -\beta V(x) + b\mathds{1}_C(x)\,\forall x\in S,\]
 then there exist $r>1,\,R<\infty$ such that for any $x\in\left\{y\in S:\,V(y)<\infty\right\}$
 \beqq
  \sum_{n=1}^\infty r^n\left\|P^n(x,\hspace{0.5mm}\cdot\hspace{0.5mm})-\pi\right\|_V\leq RV(x)
 \label{Equation B.1}
 \eeqq
 with $\|\mu\|_V:=\sup\limits_{g:\,|g|\leq V}|\mu(g)|$ for any signed measure $\mu$ defined on $(S,\scrb(S))$.
\label{Theorem B.4}
\end{satz}
%
%
\section*{Acknowledgements}
The authors wish to thank the editor, Thomas Mikosch, the associate editor and the referee for their helpful comments which considerably improved this paper. 
Moreover Florian Fuchs gratefully acknowledges the support of the TUM Graduate School's International School of Applied Mathematics at the Technische Universit\"at
M\"unchen. In addition, Florian Fuchs and Robert Stelzer are grateful for financial support of the TUM Institute for Advanced Study funded by the German Excellence
Initiative.
%

\begin{thebibliography}{10}

\bibitem{Andrews1984}
{\sc Andrews, D. W.~K.}
\newblock Non-{S}trong {M}ixing {A}utoregressive {P}rocesses.
\newblock {\em J. Appl. Probab. 21\/} (1984), 930--934.

\bibitem{Basraketal2002}
{\sc Basrak, B., Davis, R.~A., and Mikosch, T.}
\newblock Regular variation of {GARCH} processes.
\newblock {\em Stochastic Process. Appl. 99\/} (2002), 95--115.

\bibitem{Bauwensetal2006}
{\sc Bauwens, L., Laurent, S., and Rombouts, J. V.~K.}
\newblock Multivariate {GARCH} models: A survey.
\newblock {\em J. Appl. Econometrics 21\/} (2006), 79--109.

\bibitem{BENALGSET}
{\sc Benedetti, R., and Risler, J.-J.}
\newblock {\em Real algebraic and semi-algebraic sets}.
\newblock Hermann, Paris, 1990.

\bibitem{Bickeletal1999}
{\sc Bickel, P.~J., and B\"uhlmann, P.}
\newblock A new mixing notion and functional central limit theorems for a sieve
  bootstrap in time series.
\newblock {\em Bernoulli 5\/} (1999), 413--446.

\bibitem{Bollerslev1986}
{\sc Bollerslev, T.}
\newblock Generalized autoregeressive conditional heteroskedasticity.
\newblock {\em J. Econometrics 31\/} (1986), 307--327.

\bibitem{Bougeroletal1992b}
{\sc Bougerol, P., and Picard, N.}
\newblock Stationarity of {GARCH} processes and of some nonnegative time
  series.
\newblock {\em J. Econometrics 52\/} (1992), 115--127.

\bibitem{BOUGARCH}
{\sc Boussama, F.}
\newblock {\em Ergodicit\'e, m\'elange et estimation dans les mod\`eles {GARCH}}.
\newblock PhD thesis, Universit\'e {P}aris 7, 1998.

\bibitem{Boussama2006}
{\sc Boussama, F.}
\newblock Ergodicit\'e des cha\^ines de {M}arkov \`a valeurs dans une
  vari\'et\'e alg\'ebrique: application aux mod\`eles {GARCH} multivari\'es.
\newblock {\em C. R. Math. Acad. Sci. Paris 343\/} (2006), 275--278.

\bibitem{ComteLieberman2003}
{\sc Comte, F., and Lieberman, O.}
\newblock Asymptotic theory for multivariate {GARCH} processes.
\newblock {\em J. Multivariate Anal. 84\/} (2003), 61--84.

\bibitem{Davydov1973}
{\sc Davydov, Y.~A.}
\newblock Mixing conditions for {M}arkov chains.
\newblock {\em Theory Probab. Appl. 18\/} (1973), 312--328.

\bibitem{DIEELE}
{\sc Dieudonn\'e, J.}
\newblock {\em \'El\'ements d'analyse}, 2nd~ed.
\newblock t. III. Gauthier-Villars, Paris, 1974.

\bibitem{Doukhan1994}
{\sc Doukhan, P.}
\newblock {\em Mixing}, vol.~85 of {\em Lecture Notes in Statistics}.
\newblock Springer, New York, 1994.

\bibitem{Doukhanetal1994}
{\sc Doukhan, P., Massart, P., and Rio, E.}
\newblock The functional central limit theorem for strongly mixing processes.
\newblock {\em Ann. Inst. H. Poincar\'e Sect. B 30\/} (1994), 63--82.

\bibitem{Doukhanetal2008}
{\sc Doukhan, P., and Wintenberger, O.}
\newblock Weakly dependent chains with infinite memory.
\newblock {\em Stochastic Process. Appl. 118\/} (2008), 1997--2013.

\bibitem{Engle1982}
{\sc Engle, R.~F.}
\newblock Autoregressive conditional heteroskedasticity with estimates of the
  variance of {U}nited {K}ingdom inflation.
\newblock {\em Econometrica 50\/} (1982), 987--1008.

\bibitem{ENGMUL}
{\sc Engle, R.~F., and Kroner, K.~F.}
\newblock Multivariate {S}imultaneous {G}eneralized {ARCH}.
\newblock {\em Econometric Theory 11\/} (1995), 122--150.

\bibitem{Hafner2003b}
{\sc Hafner, C.~M.}
\newblock Fourth moment structure of multivariate {GARCH} models.
\newblock {\em Journal of Financial Econometrics 1\/} (2003), 26--54.

\bibitem{Haeggstroem2005}
{\sc H\"aggstr\"om, O.}
\newblock On the central limit theorem for geometrically ergodic {M}arkov
  chains.
\newblock {\em Probab. Theory Related Fields 132\/} (2005), 74--82.

\bibitem{Herrndorf1983}
{\sc Herrndorf, N.}
\newblock Stationary {S}trongly {M}ixing {S}equences {N}ot {S}atisfying the
  {C}entral {L}imit {T}heorem.
\newblock {\em Ann. Probab. 11\/} (1983), 809--813.

\bibitem{HORMAT}
{\sc Horn, R.~A., and Johnson, C.~R.}
\newblock {\em Matrix {A}nalysis}.
\newblock Cambridge University Press, Cambridge, 1985.

\bibitem{Jones2004}
{\sc Jones, G.~L.}
\newblock On the {M}arkov chain central limit theorem.
\newblock {\em Probab. Surv. 1\/} (2004), 299--320.

\bibitem{LANALG}
{\sc Lang, S.}
\newblock {\em Algebra}, {R}ev. 3rd~ed.
\newblock Springer, 2002.

\bibitem{Lindner2009}
{\sc Lindner, A.~M.}
\newblock Stationarity, mixing, distributional properties and moments of
  {GARCH($p,q$)}-processes.
\newblock In {\em Handbook of Financial Time Series\/} (Berlin, 2009), T.~G.
  Andersen, R.~A. Davis, J.-P. Krei{\ss}, and T.~Mikosch, Eds., Springer,
  pp.~43--69.

\bibitem{MeynTweedie1992}
{\sc Meyn, S.~P., and Tweedie, R.~L.}
\newblock Stability of {M}arkovian processes {I}: Criteria for discrete-time
  chains.
\newblock {\em Adv. in Appl. Probab. 24\/} (1992), 542--574.

\bibitem{Meynetal1993}
{\sc Meyn, S.~P., and Tweedie, R.~L.}
\newblock {\em Markov Chains and Stochastic Stability}.
\newblock Springer, London, 1993.

\bibitem{MOKPROP}
{\sc Mokkadem, A.}
\newblock Propri\'et\'es de m\'elange des processus autor\'egressifs polynomiaux.
\newblock {\em Ann. Inst. H. Poincar\'e Probab. Statist. 26\/} (1990), 219--260.

\bibitem{MUMALGGEO}
{\sc Mumford, D.}
\newblock {\em Algebraic {G}eometry {I}, {C}omplex {P}rojective {V}arieties}.
\newblock Springer-Verlag, Berlin, 1976.

\bibitem{POTMAT}
{\sc Potter, J.~E.}
\newblock Matrix {Q}uadratic {S}olutions.
\newblock {\em SIAM J. Appl. Math. 14\/} (1966), 496--501.

\bibitem{Sato1999}
{\sc Sato, K.}
\newblock {\em L\'evy Processes and Infinitely Divisible Distributions},
  vol.~68 of {\em Cambridge Studies in Advanced Mathematics}.
\newblock Cambridge University Press, Cambridge, 1999.

\bibitem{SilvennoinenTeraesvirta2009}
{\sc Silvennoinen, A., and Ter\"asvirta, T.}
\newblock Multivariate {GARCH} models.
\newblock In {\em Handbook of Financial Time Series\/} (Berlin, 2009), T.~G.
  Andersen, R.~A. Davis, J.-P. Krei{\ss}, and T.~Mikosch, Eds., Springer,
  pp.~201--229.

\bibitem{STEREL}
{\sc Stelzer, R.}
\newblock On the relation between the vec and {BEKK} multivariate {GARCH}
  models.
\newblock {\em Econometric Theory 24\/} (2008), 1131--1136.

\bibitem{WARPAR}
{\sc Ward, L.~E.}
\newblock Partially {O}rdered {T}opological {S}paces.
\newblock {\em Proc. Amer. Math. Soc. 5\/} (1954), 144--161.

\end{thebibliography}
%

\end{document}